\documentclass[12pt,a4paper]{article}
\usepackage[utf8]{inputenc}
\usepackage[affil-it]{authblk}
\usepackage{amsmath}
\usepackage{amsthm}
\usepackage{bbm}
\usepackage{bm}
\usepackage[top=3cm,left=3cm,right=3cm,bottom=3cm]{geometry}
\usepackage{txfonts}
\usepackage{amsfonts}
\usepackage{amssymb}
\usepackage{hyperref}
\usepackage{enumerate}
\usepackage{enumitem}
\usepackage{graphicx}
\usepackage{tikz-cd}

\usepackage{color}

\usepackage{tikz,pgfplots}
\usetikzlibrary{arrows}
\newcommand{\pnt}[3][black]{%
	\begin{scope}[shift={#2}];
		\fill[color=#1,shift only] (0,0) circle(#3);
	\end{scope}}

\DeclareMathOperator{\sign}{sign}
\newcommand{\p}{\partial}
\newcommand{\id}{\textrm{id}}

\newcommand{\NN}{\mathbb N}
\newcommand{\IN}{\mathbb Z}
\newcommand{\RN}{\mathbb R}

\newcommand{\TT}{\mathbb T}

\allowdisplaybreaks[1]

\makeindex

\theoremstyle{plain}
\newtheorem{thm}{Theorem}[section]
\newtheorem{lem}[thm]{Lemma}

\theoremstyle{definition}

\theoremstyle{remark}
\newtheorem{rmk}[thm]{Remark}

\numberwithin{equation}{section}

\begin{document}

\title{Escaping orbits are rare in the quasi-periodic Littlewood boundedness problem}
\author{Henrik Schließauf
	\thanks{
		Electronic address: \texttt{hschlies@math.uni-koeln.de}
	}
}
\affil{
	Universität zu Köln, Mathematisches Institut, \\
	Weyertal 86-90, 50931 Köln, Germany
}
\date{November 29, 2018}

\maketitle

\begin{abstract}
	\noindent 
	We study the superlinear oscillator equation $\ddot{x}+ \lvert x \rvert^{\alpha-1}x = p(t)$ for $\alpha\geq 3$, where $p$ is a quasi-periodic forcing with no Diophantine condition on the frequencies and show that typically the set of initial values leading to solutions $x$ such that $\lim_{t\to\infty} (\lvert x(t) \rvert + \lvert \dot{x}(t) \rvert) = \infty$ has Lebesgue measure zero, provided the starting energy $\lvert x(t_0) \rvert + \lvert \dot{x}(t_0) \rvert$ is sufficiently large. 
\end{abstract}

\section{Introduction}

The dynamics of the Duffing-type equation
\begin{equation} \label{duffing eq}
\ddot x + G'(x) = p(t),
\end{equation}
have been studied extensively due to its relevance as a model for the motion of a classical particle in a one-dimensional potential field $G(x)$ affected by an external time-dependent force $p(t)$. In the 1960's, Littlewood \cite{Littlewood_unbounded_solutions} asked whether solutions of (\ref{duffing eq}) stay bounded in the $(x,\dot{x})$-phase space if either
\begin{align*}
&(i) \; G'(x)/x \to +\infty \;\; \text{as} \;\; x \to \pm \infty \\
\text{or} \;\; &(ii)  \; \sign (x) \cdot G'(x) \to +\infty \;\; \text{and} \;\;  G'(x)/x \to 0 \;\; \text{as} \;\; x \to \pm \infty.
\end{align*}
Despite it's harmless appearance, this question turned out to be a quite delicate matter. Whether some resonance phenomena occur, obviously does not only depend on the growth of $G$, but also on the properties of $p$ with respect to regularity and (quasi)-periodicity. The most investigated case is that of a time-periodic forcing $p$. The first affirmative contribution in that regard is due to Morris \cite{Morris_a_case_of_boundedness}, who showed the boundedness of all solutions to 
\begin{equation*}
\ddot x + 2x^3 = p(t),
\end{equation*}
where $p$ is continuous and periodic. Later, Dieckerhoff and Zehnder \cite{dieckerhoff_zehnder_boundedness_via_twisttheorem} were able to show the same for
\begin{equation*}
\ddot x + x^{2n+1} + \sum_{j=0}^{2n} p_j(t) x^{j} = 0,
\end{equation*}
where $n\in\NN$ and $p_j \in \mathcal{C}^{\infty}$ are $1$-periodic. In the following years, this result was improved by several authors (see \cite{Bin_1989}, \cite{Laederich_1991}, \cite{levi_quasiperiodic_motions},\cite{Norris_1992},\cite{levi_zehnder_quasiperiodic} and the references therein). If however the periodicity condition is dropped, Littlewood \cite{Littlewood_unbounded_solutions} himself showed that for any odd potential $G$ satisfying the super-/sublinearity condition there exists a bounded forcing $p$ leading to at least one unbounded trajectory. Later, Ortega \cite{Ortega_2005} was able to prove in a more general context that for any given $\mathcal{C}^2$-potential one can find an arbitrarily small $p\in \mathcal{C}^\infty$ such that most initial conditions (in the sense of a residual set) correspond to unbounded solutions of (\ref{duffing eq}). Even in the time-periodic case Littlewood \cite{Littlewood_unbounded_solutions_periodic_p} constructed $G \in \mathcal{C}^\infty$ and a periodic $p$ such that there is at least one unbounded solution. (Actually both \cite{Littlewood_unbounded_solutions} and \cite{Littlewood_unbounded_solutions_periodic_p} contain a computational mistake; see \cite{Levi_counterexample,Long} for corrections.) Let us also mention \cite{Zharnitsky_1997}, where Zharnitsky improved the latter result for the superlinear case such that the periodic $p$ can be chosen continuously. These counterexamples show that besides periodicity and regularity assumptions on $p$ an additional hypothesis on $G$ is needed if one hopes for boundedness of all solutions. Indeed, all positive results mentioned above suppose the monotone growth of $G'(x)/x$. This condition guarantees the monotonicity of the corresponding Poincaré map and thus enables the authors to use KAM theory. \\
We also want to point out the related problem of the so called Fermi-Ulam ``ping pong'' \cite{Fermi_on_the_origin,ulam1961}. The latter is a model for a particle bouncing elastically between periodically moving walls. In \cite{Laederich_1991}, it was first shown that for sufficiently regular motions the velocity of the particle stays bounded for all time and many results followed ever after.\\
In the last twenty years a wealth of works on the Littlewood boundedness problem has been published, including the sublinear, semilinear and other cases (see \cite{Kuepper_1999,Li_2001,Wang_2009,Liu_2009} and the references therein for some examples). Since those are far too many to be presented here, let us focus on the superlinear oscillator equation
\begin{equation} \label{DGL 0}
\ddot x + \lvert x \rvert^{\alpha-1}x = p(t),
\end{equation}
where $\alpha \geq 3$. In \cite{levi_zehnder_quasiperiodic}, Levi and Zehnder were able to show that for a quasi-periodic forcing $p$ all solutions are bounded, if the frequencies of $p$ satisfy a diophantine condition. In the present paper we shall omit this restriction on the frequencies and investigate the resulting long term behavior of solutions. But since in that case the tools related to invariant curve theorems are not available, one needs a different approach. In \cite{kunze_ortega_ping_pong}, Kunze and Ortega presented a technique applicable in the above situation. Using a refined version of Poincaré's recurrence theorem due to Dolgopyat \cite{dolgopyat}, they proved that under appropriate conditions almost all orbits of a certain successor map $f$ are recurrent. In particular, they used this theorem to show that quasi-periodic forcing functions $p$ lead to recurrent orbits in the Fermi-Ulam ``ping pong''. Here, we want to do the same for (\ref{DGL 0}). If $x(t)$ denotes a solution to this equation, we consider the map 
\begin{equation*}
\psi:(v_0,t_0)\mapsto(v_1,t_1),
\end{equation*}
which sends the time $t_0$ of a zero with negative derivative $v_0=\dot{x}(t_0)$ to the subsequent zero $t_1$ of this form and its corresponding velocity $v_1=\dot{x}(t_1)<0$. This map will be well defined for $\lvert v_0 \rvert$ sufficiently large, since in this case the corresponding solution oscillates quickly.
\begin{figure*}
	\begin{center}
		\begin{tikzpicture}[xscale=.5,yscale=.23]

		\draw[thick] (0 ,  -5.0000)
		(-0.0314  , -5.0000)--
		(-0.0628  , -4.9998)--
		(-0.0942  , -4.9996)--
		(-0.1257  , -4.9992)--
		(-0.1571  , -4.9988)--
		(-0.1885 ,  -4.9982)--
		(-0.2199 ,  -4.9975)--
		(-0.2513  , -4.9968)--
		(-0.2827  , -4.9958)--
		(-0.3141 ,  -4.9948)--
		(-0.3454,   -4.9935)--
		(-0.3768,   -4.9921)--
		(-0.4082 ,  -4.9906)--
		(-0.4395 ,  -4.9888)--
		(-0.4709 ,  -4.9868)--
		(-0.5022 ,  -4.9846)--
		(-0.5335 ,  -4.9822)--
		(-0.5648,   -4.9795)--
		(-0.5961 ,  -4.9765)--
		(-0.6273  , -4.9732)--
		(-0.6586   ,-4.9696)--
		(-0.6898,   -4.9656)--
		(-0.7210 ,  -4.9613)--
		(-0.7521  , -4.9565)--
		(-0.7832   ,-4.9514)--
		(-0.8143,   -4.9458)--
		(-0.8454 ,  -4.9399)--
		(-0.8764  , -4.9335)--
		(-0.9074 ,  -4.9266)--
		(-0.9384,   -4.9191)--
		(-0.9693 ,  -4.9111)--
		(-1.0002  , -4.9024)--
		(-1.0310 ,  -4.8931)--
		(-1.0617  , -4.8830)--
		(-1.0924,   -4.8722)--
		(-1.1230,   -4.8606)--
		(-1.1535,   -4.8481)--
		(-1.1840,   -4.8347)--
		(-1.2143,   -4.8204)--
		(-1.2446,   -4.8051)--
		(-1.2747,   -4.7887)--
		(-1.3048,   -4.7714)--
		(-1.3347,   -4.7529)--
		(-1.3645,   -4.7332)--
		(-1.3942,   -4.7124)--
		(-1.4237,   -4.6904)--
		(-1.4531,   -4.6672)--
		(-1.4824,   -4.6426)--
		(-1.5114,   -4.6167)--
		(-1.5403,   -4.5895)--
		(-1.5691,   -4.5609)--
		(-1.5976,   -4.5308)--
		(-1.6260,   -4.4993)--
		(-1.6541,   -4.4663)--
		(-1.6821,   -4.4318)--
		(-1.7098,   -4.3957)--
		(-1.7373,   -4.3580)--
		(-1.7645,   -4.3188)--
		(-1.7915,   -4.2779)--
		(-1.8182,   -4.2353)--
		(-1.8447,   -4.1910)--
		(-1.8708,   -4.1450)--
		(-1.8967,   -4.0972)--
		(-1.9223,   -4.0476)--
		(-1.9475,   -3.9963)--
		(-1.9725,   -3.9431)--
		(-1.9971,   -3.8881)--
		(-2.0213,   -3.8311)--
		(-2.0452,   -3.7723)--
		(-2.0687,   -3.7116)--
		(-2.0918,   -3.6489)--
		(-2.1145,   -3.5842)--
		(-2.1368,   -3.5176)--
		(-2.1587,   -3.4489)--
		(-2.1802,   -3.3783)--
		(-2.2012,   -3.3055)--
		(-2.2217,   -3.2307)--
		(-2.2418,   -3.1539)--
		(-2.2614,   -3.0750)--
		(-2.2805,   -2.9941)--
		(-2.2991,   -2.9112)--
		(-2.3172,   -2.8264)--
		(-2.3347,   -2.7398)--
		(-2.3517,   -2.6513)--
		(-2.3682,   -2.5610)--
		(-2.3840,   -2.4690)--
		(-2.3993,   -2.3754)--
		(-2.4140,   -2.2801)--
		(-2.4281,   -2.1832)--
		(-2.4415,   -2.0849)--
		(-2.4544,   -1.9850)--
		(-2.4665,   -1.8838)--
		(-2.4781,   -1.7812)--
		(-2.4890,   -1.6773)--
		(-2.4992,   -1.5723)--
		(-2.5087,   -1.4660)--
		(-2.5176,   -1.3587)--
		(-2.5258,   -1.2503)--
		(-2.5333,   -1.1410)--
		(-2.5400,   -1.0308)--
		(-2.5461,   -0.9197)--
		(-2.5515,   -0.8079)--
		(-2.5562,   -0.6954)--
		(-2.5602,   -0.5823)--
		(-2.5634,   -0.4687)--
		(-2.5659,   -0.3546)--
		(-2.5678,   -0.2402)--
		(-2.5688,   -0.1254)--
		(-2.5692,   -0.0105)--
		(-2.5689,    0.1047)--
		(-2.5678,    0.2198)--
		(-2.5660,    0.3350)--
		(-2.5660,    0.3350) --
		(-2.5635,0.45)--
		(-2.5603,0.5648)--
		(-2.5564,0.6794)--
		(-2.5517,0.7935)--
		(-2.5464,0.9072)--
		(-2.5403,1.0202)--
		(-2.5336,1.1326)--
		(-2.5261,1.2441)--
		(-2.5179,1.3547)--
		(-2.5091,1.4644)--
		(-2.4996,1.573)--
		(-2.4894,1.6805)--
		(-2.4785,1.7868)--
		(-2.4669,1.8919)--
		(-2.4548,1.9957)--
		(-2.4419,2.0981)--
		(-2.4285,2.1991)--
		(-2.4144,2.2986)--
		(-2.3996,2.3966)--
		(-2.3843,2.493)--
		(-2.3684,2.5878)--
		(-2.3519,2.681)--
		(-2.3348,2.7724)--
		(-2.3171,2.8621)--
		(-2.2989,2.9501)--
		(-2.2801,3.0362)--
		(-2.2607,3.1206)--
		(-2.2409,3.2031)--
		(-2.2205,3.2838)--
		(-2.1996,3.3626)--
		(-2.1782,3.4395)--
		(-2.1564,3.5145)--
		(-2.134,3.5876)--
		(-2.1112,3.6589)--
		(-2.088,3.7282)--
		(-2.0643,3.7956)--
		(-2.0402,3.8612)--
		(-2.0157,3.9248)--
		(-1.9908,3.9866)--
		(-1.9655,4.0466)--
		(-1.9399,4.1047)--
		(-1.9139,4.161)--
		(-1.8875,4.2155)--
		(-1.8609,4.2682)--
		(-1.8339,4.3192)--
		(-1.8066,4.3686)--
		(-1.779,4.4162)--
		(-1.7511,4.4622)--
		(-1.7229,4.5066)--
		(-1.6945,4.5495)--
		(-1.6658,4.5908)--
		(-1.6368,4.6307)--
		(-1.6076,4.6692)--
		(-1.5781,4.7062)--
		(-1.5485,4.7419)--
		(-1.5186,4.7763)--
		(-1.4884,4.8094)--
		(-1.4581,4.8413)--
		(-1.4276,4.872)--
		(-1.3969,4.9016)--
		(-1.366,4.93)--
		(-1.335,4.9574)--
		(-1.3037,4.9837)--
		(-1.2723,5.009)--
		(-1.2408,5.0334)--
		(-1.2091,5.0569)--
		(-1.1772,5.0795)--
		(-1.1452,5.1012)--
		(-1.1131,5.1222)--
		(-1.0809,5.1425)--
		(-1.0485,5.162)--
		(-1.016,5.1808)--
		(-0.9834,5.199)--
		(-0.9507,5.2167)--
		(-0.9178,5.2337)--
		(-0.8849,5.2503)--
		(-0.8518,5.2664)--
		(-0.8187,5.2821)--
		(-0.7855,5.2973)--
		(-0.7521,5.3123)--
		(-0.7187,5.3269)--
		(-0.6852,5.3411)--
		(-0.6516,5.3551)--
		(-0.6179,5.3687)--
		(-0.5842,5.3822)--
		(-0.5503,5.3954)--
		(-0.5164,5.4084)--
		(-0.4824,5.4212)--
		(-0.4483,5.4339)--
		(-0.4141,5.4465)--
		(-0.3798,5.459)--
		(-0.3455,5.4714)--
		(-0.3111,5.4837)--
		(-0.2766,5.496)--
		(-0.242,5.5082)--
		(-0.2073,5.5204)--
		(-0.1726,5.5325)--
		(-0.1378,5.5447)--
		(-0.1029,5.5568)--
		(-0.0679,5.5689)--
		(-0.0329,5.581)--
		(0.0023,5.5932)--
		(0.0375,5.6053)--
		(0.0727,5.6174)--
		(0.1081,5.6295)--
		(0.1435,5.6416)--
		(0.179,5.6536)--
		(0.2146,5.6656)--
		(0.2502,5.6776)--
		(0.2859,5.6895)--
		(0.3217,5.7013)--
		(0.3576,5.7131)--
		(0.3935,5.7247)--
		(0.4295,5.7362)--
		(0.4656,5.7476)--
		(0.5017,5.7588)--
		(0.5379,5.7697)--
		(0.5742,5.7805)--
		(0.6105,5.791)--
		(0.6469,5.8013)--
		(0.6834,5.8112)--
		(0.7199,5.8208)--
		(0.7565,5.83)--
		(0.7931,5.8387)--
		(0.8299,5.8472)--
		(0.8666,5.8552)--
		(0.9034,5.8628)--
		(0.9403,5.8698)--
		(0.9772,5.8763)--
		(1.0142,5.882)--
		(1.0512,5.887)--
		(1.0882,5.8912)--
		(1.1252,5.8945)--
		(1.1623,5.8968)--
		(1.1993,5.8981)--
		(1.2364,5.8983)--
		(1.2735,5.8973)--
		(1.3105,5.8951)--
		(1.3476,5.8916)--
		(1.3846,5.8868)--
		(1.4216,5.8805)--
		(1.4585,5.8726)--
		(1.4954,5.8632)--
		(1.5322,5.8522)--
		(1.5689,5.8394)--
		(1.6056,5.8248)--
		(1.6421,5.8084)--
		(1.6786,5.7901)--
		(1.7149,5.7698)--
		(1.751,5.7475)--
		(1.7871,5.723)--
		(1.8229,5.6963)--
		(1.8586,5.6674)--
		(1.8941,5.6361)--
		(1.9294,5.6024)--
		(1.9645,5.5663)--
		(1.9994,5.5277)--
		(2.034,5.4865)--
		(2.0683,5.4426)--
		(2.1023,5.396)--
		(2.1361,5.3466)--
		(2.1695,5.2943)--
		(2.2026,5.2392)--
		(2.2353,5.181)--
		(2.2677,5.1198)--
		(2.2996,5.0554)--
		(2.3312,4.9879)--
		(2.3623,4.9171)--
		(2.3929,4.843)--
		(2.4231,4.7655)--
		(2.4528,4.6846)--
		(2.482,4.6001)--
		(2.5106,4.5121)--
		(2.5387,4.4207)--
		(2.5662,4.3257)--
		(2.5931,4.2274)--
		(2.6194,4.1256)--
		(2.6451,4.0205)--
		(2.67,3.9119)--
		(2.6943,3.8)--
		(2.7179,3.6849)--
		(2.7407,3.5665)--
		(2.7627,3.4449)--
		(2.784,3.3201)--
		(2.8045,3.1923)--
		(2.8241,3.0616)--
		(2.8429,2.9279)--
		(2.8609,2.7913)--
		(2.8779,2.6521)--
		(2.8941,2.5102)--
		(2.9094,2.3658)--
		(2.9238,2.2189)--
		(2.9372,2.0697)--
		(2.9497,1.9184)--
		(2.9612,1.765)--
		(2.9718,1.6097)--
		(2.9814,1.4525)--
		(2.99,1.2938)--
		(2.9976,1.1335)--
		(3.0042,0.972)--
		(3.0098,0.8093)--
		(3.0144,0.6455)--
		(3.0179,0.481);
		\draw[thick] (3.0179,0.481)--
		(3.0204,0.3158)--
		(3.0219,0.1501)--
		(3.0223,-0.0158)--
		(3.0217,-0.1818)--
		(3.0201,-0.3476)--
		(3.0174,-0.5131)--
		(3.0137,-0.6781)--
		(3.009,-0.8423)--
		(3.0032,-1.0057)--
		(2.9964,-1.1681)--
		(2.9886,-1.3292)--
		(2.9798,-1.489)--
		(2.9699,-1.6472)--
		(2.9591,-1.8038)--
		(2.9473,-1.9586)--
		(2.9345,-2.1115)--
		(2.9208,-2.2622)--
		(2.9061,-2.4108)--
		(2.8904,-2.5571)--
		(2.8739,-2.7009)--
		(2.8564,-2.8422)--
		(2.8381,-2.9809)--
		(2.8189,-3.1168)--
		(2.7989,-3.2499)--
		(2.778,-3.38)--
		(2.7563,-3.5072)--
		(2.7338,-3.6313)--
		(2.7106,-3.7522)--
		(2.6867,-3.8699)--
		(2.662,-3.9844)--
		(2.6366,-4.0956)--
		(2.6105,-4.2035)--
		(2.5838,-4.3081)--
		(2.5564,-4.4093)--
		(2.5284,-4.5073)--
		(2.4998,-4.6019)--
		(2.4706,-4.6933)--
		(2.4408,-4.7813)--
		(2.4105,-4.866)--
		(2.3797,-4.9475)--
		(2.3484,-5.0258)--
		(2.3166,-5.1008)--
		(2.2843,-5.1726)--
		(2.2516,-5.2414)--
		(2.2185,-5.307)--
		(2.1849,-5.3696)--
		(2.151,-5.4292)--
		(2.1167,-5.486)--
		(2.082,-5.5398)--
		(2.0471,-5.5909)--
		(2.0118,-5.6392)--
		(1.9762,-5.6849)--
		(1.9403,-5.7281)--
		(1.9042,-5.7688)--
		(1.8678,-5.8072)--
		(1.8312,-5.8432)--
		(1.7944,-5.8771)--
		(1.7574,-5.9088)--
		(1.7201,-5.9385)--
		(1.6827,-5.9661)--
		(1.6452,-5.9918)--
		(1.6075,-6.0157)--
		(1.5696,-6.0378)--
		(1.5316,-6.0582)--
		(1.4935,-6.0771)--
		(1.4552,-6.0945)--
		(1.4169,-6.1104)--
		(1.3785,-6.1249)--
		(1.3399,-6.1381)--
		(1.3013,-6.1502)--
		(1.2626,-6.161)--
		(1.2239,-6.1708)--
		(1.1851,-6.1796)--
		(1.1462,-6.1874)--
		(1.1073,-6.1944)--
		(1.0684,-6.2005)--
		(1.0294,-6.2058)--
		(0.9904,-6.2105)--
		(0.9514,-6.2145)--
		(0.9123,-6.2179)--
		(0.8732,-6.2208)--
		(0.8341,-6.2231)--
		(0.795,-6.2251)--
		(0.7559,-6.2267)--
		(0.7168,-6.228)--
		(0.6776,-6.2289)--
		(0.6385,-6.2297)--
		(0.5994,-6.2302)--
		(0.5602,-6.2305)--
		(0.5211,-6.2307)--
		(0.482,-6.2307)--
		(0.4428,-6.2307)--
		(0.4037,-6.2306)--
		(0.3646,-6.2305)--
		(0.3254,-6.2303)--
		(0.2863,-6.2301)--
		(0.2472,-6.23)--
		(0.208,-6.2299)--
		(0.1689,-6.2299)--
		(0.1297,-6.2299)--
		(0.0906,-6.23)--
		(0.0514,-6.2301)--
		(0.0122,-6.2304);
		
		\begin{scope}[>=latex]
		\draw[->] (0,-7)--(0,7.5) node[left] {$\dot{x}$};
		\draw[->] (-4,0)--(4,0)node[below] {$x$};
		\end{scope}
		
		\begin{scope}[>=stealth]
		\draw[->,very thick] (-2.5692,   -0.0105)--(-2.5689,    0.1047) ;
		\draw[->,very thick] (3.0219,0.1501)--(3.0223,-0.0158) ;
		\end{scope}
		
		\node at(.5 ,  -5.0000)      {$v_0$};
		\node at (-.5,-6.2304)      {$v_1$};
		\pnt[black]{(0,-5)}{0.06}
		\pnt[black]{(0,-6.2304)}{0.06}
		\end{tikzpicture}
	\end{center}
	\caption{For large energies the trajectory spins clockwise around the origin}
\end{figure*}
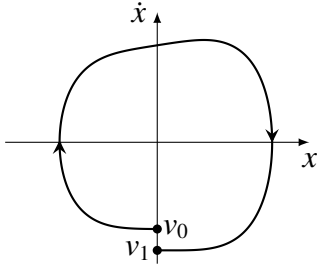
Defining the forward iterates $(v_n,t_n) = \psi^n(v_0,t_0)$ for $n\in \NN$, we study the escaping set
\begin{equation} \label{def escaping set}
E = \{ (v_0,t_0) : \lim_{n\to\infty} v_n = -\infty \}.
\end{equation}
Especially, for any solution $x$ such that $\lim_{t\to\infty} (\lvert x(t) \rvert + \lvert \dot{x}(t) \rvert) = \infty$ there is a time $t_0$ with $(\dot{x}(t_0),t_0) \in E$. 
Now, let us state the main result. Note however, that due to the vague definition of $\psi$ this formulation of the theorem will still be somewhat imprecise. A rigorous definition of the successor map $\psi$ and its domain $(-\infty,v_*)\times \RN$ will be given only in the last section. We refer the reader to this part and Theorem \ref{main theorem} below for the complete statement.
\begin{thm} \label{main thm short}
	Let $\bm{p}\in \mathcal{C}^4(\TT^N)$ generate the family of forcing functions
	\begin{equation} \label{def quasiperiodic p}
	p_{\bar{\Theta}}(t)=\bm{p}(\overline{\theta_1+t\omega_1}, \ldots, \overline{\theta_N+t\omega_N}), \;\; \bar{\Theta}=(\bar\theta_1,\ldots,\bar\theta_N)\in\TT^N,
	\end{equation}
	for fixed rationally independent frequencies $\omega_1,\ldots,\omega_N>0$. Let $(v_n,t_n)_{n \in \NN_0} = (\psi^n(v_0,t_0))_{n \in \NN_0}$ denote a generic complete forward orbit associated to system (\ref{DGL 0}) with $p$ replaced by $p_{\bar{\Theta}}$ and let $\mathcal{E}_{\bar{\Theta}}$ denote the corresponding escaping set (\ref{def escaping set}). Then, for almost all $\bar{\Theta} \in \TT^N$, the set $\mathcal{E}_{\bar{\Theta}}$ has Lebesgue measure zero.
\end{thm}
Here $\TT$ stands for the torus $\RN / \IN$ and $\mathcal{C}^k(\TT^N)$ denotes the space of functions $\bm{p}\in \mathcal{C}^k(\RN^N)$ that are $1$-periodic in each argument.
\begin{rmk}
	For $\alpha>3$ the same holds, if only $\bm{p} \in \mathcal{C}^2(\TT^N)$ is imposed. Just Theorem \ref{thm ham transformation} of the second section requires higher regularity and the latter is needed only for $\alpha=3$.
\end{rmk}
Let us give a short outline of the paper. As mentioned above, the theorem about escaping sets by Kunze and Ortega is the main tool used therein. It deals with certain quasi-periodic successor maps $(t_0,r_0) \mapsto (t_1,r_1)$ on $\RN\times (0,\infty)$ and associated maps $f(\bar{\Theta}_0,r_0) = (\bar{\Theta}_1,r_1)$ on $\TT^N \times (0,\infty)$. If one can find an adiabatic invariant-type function $W$, such that
\begin{equation*}
	W(f(\bar{\Theta},r)) \leq W(\bar{\Theta},r) + k(r),
\end{equation*}
where $k:(0,\infty) \to \RN$ is a decreasing and bounded function such that $\lim_{r\to\infty} k(r)=0$, then most orbits are recurrent. Section \ref{section escaping sets} is dedicated to presenting the terminology and general setup needed in order to state this theorem. For a more detailed description as well as the proof, we refer the reader to \cite{kunze_ortega_ping_pong}.\\
In the subsequent section we basically follow \cite{kunze_ortega_stability_estimates} by developing three diffeomorphisms that transform the underlying equation (\ref{DGL 0}) into a more convenient form suitable for the setup. First we restate the problem in terms of action-angle coordinates $(\bar{\vartheta},r)$ associated to the unperturbed system, as was already done by Morris \cite{Morris_a_case_of_boundedness}. Afterwards, we consider a transformation $\mathcal{S}$ by which the time $t=\phi$ and the old Hamiltonian become the new conjugate variables, whereas the old symplectic angle $\vartheta=\tau$ is chosen as the new independent variable. This trick was refined by Levi \cite{levi_quasiperiodic_motions} in the context of Littlewood's boundedness problem. In the third part of this section, we present a canonical change of variables $\mathcal{T}$ developed by Kunze and Ortega \cite{kunze_ortega_stability_estimates} that preserves the structure of the Hamiltonian, while making the new momentum coordinate $\mathcal{I}$ an adiabatic invariant in the sense of the estimate above. Those transformations can be shortly illustrated as follows:
\begin{equation*}
(x,\dot x;t) \overset{\mathcal{R}}{\rightarrow} (\bar{\vartheta},r;t) \hookrightarrow  (\vartheta,r;t) \overset{\mathcal{S}}{\rightarrow} (\phi,I;\tau) \overset{\mathcal{T}}{\rightarrow} (\varphi,\mathcal{I};\tau)
\end{equation*}
The estimate is confirmed in the next section, where we show that the time-$2\pi$ map $\Phi(\varphi_0,\mathcal{I}_0)$ of the resulting system defines an area-preserving successor map. Subsequently we prove that all three transformations retain the quasi-periodic structure if $p_{{\Theta}}(t)$ from (\ref{def quasiperiodic p}) is taken as the forcing function. This links the developed set of coordinates $(\varphi,\mathcal{I})$ to the setup of the second section. Finally, the map $\psi$ is defined properly, leading to a restatement of the main result. We end the paper with a proof of this theorem, where it is shown that initial conditions in $\mathcal{E}_{\bar{\Theta}}$ correspond to non-recurrent orbits $(\varphi_n,\mathcal{I}_n)_{n\in\NN}$.

\section{A theorem about escaping sets} \label{section escaping sets}
In this section we want to state the aforementioned theorem about escaping sets of Kunze and Ortega. But first we need to introduce some basic notation and terminology.

\subsection{Measure-preserving embeddings}
We will write $\TT=\RN / \IN$ for the standard torus and denote the class in $\TT^N$ corresponding to a vector $\Theta=(\theta_1,\ldots,\theta_N) \in \RN^N$ by $\bar{\Theta}=(\bar\theta_1,\ldots,\bar\theta_N)$, where $\bar\theta_j = \theta_j + \IN$. Similarly, $\bar{\vartheta}$ will indicate the class in $\mathbb{S}^1 = \RN / 2\pi\IN$ associated to some $\vartheta \in \RN$. Denote by $\mu_{\TT^N}$ the unique Haar measure such that $\mu_{\TT^N}(\TT^N)=1$.
From now on we will consider functions
\begin{equation*}
f:\mathcal{D}\subset \TT^N\times(0,\infty) \to \TT^N\times(0,\infty),
\end{equation*}
where $\mathcal{D}$ is an open set. We will call such a function \textit{measure-preserving embedding}, if $f$ is continuous, injective and furthermore
\begin{equation*}
(\mu_{\TT^N} \otimes \lambda) (f(\mathcal{B}))= (\mu_{\TT^N} \otimes \lambda) (\mathcal{B})
\end{equation*}
holds for all Borel sets $\mathcal{B}\subset \mathcal{D}$, where $\lambda$ denotes the Lebesgue measure on $\RN$. It is easy to show that under these conditions, $f:\mathcal{D} \to \tilde{\mathcal{D}}$ is a homeomorphism, where $\tilde{\mathcal{D}}=f(\mathcal{D})$. \\
Since we want to use the iterations of $f$, we have to carefully construct a suitable domain on which these forward iterations are well-defined. We initialize $\mathcal{D}_1 = \mathcal{D}, \;\; f^1=f$ and set
\begin{equation*}
\mathcal{D}_{n+1}=f^{-1}(\mathcal{D}_{n}), \;\; f^{n+1}= f^n \circ f \; \text{ for  }  \; n \in \NN.
\end{equation*}
This way $f^n$ is well-defined on $\mathcal{D}_n$. Clearly, $f^n$ is a measure-preserving embedding as well. Also, it can be shown inductively that $\mathcal{D}_{n+1}=\{(\bar{\Theta},r)\in\mathcal{D}:f(\bar{\Theta},r),\ldots,f^n(\bar{\Theta},r)\in \mathcal{D} \}$
and therefore $\mathcal{D}_{n+1}\subset \mathcal{D}_n \subset \mathcal{D}$ for all $n \in \NN$.
Initial conditions in the set
\begin{equation*}
\mathcal{D}_\infty = \bigcap\limits_{n=1}^\infty \mathcal{D}_n \subset \TT^N\times(0,\infty)
\end{equation*}
correspond to complete forward orbits, i.e. if $(\bar{\Theta}_0,r_0)\in \mathcal{D}_\infty$, then
\begin{equation*}
(\bar{\Theta}_n,r_n)=f^n(\bar{\Theta}_0,r_0)
\end{equation*}
is defined for all $n \in \NN$. It could however happen that $\mathcal{D}_\infty = \emptyset$ or $\mathcal{D}_{n} = \emptyset$ for some $n\geq 2$.

\subsection{Quasi-periodic functions} \label{subsection quasi-periodic functions}

Let $\omega_1,\ldots,\omega_N>0$ be rationally independent and consider the map
\begin{equation*}
\iota:\RN \to \TT^N, \;\; \iota(t)=(\overline{t\omega_1},\ldots,\overline{t\omega_N}).
\end{equation*}
For $N>1$ this homomorphism is injective and the image $\iota(\RN)\subset \TT^N$ is dense. If $N=1$, then $\iota$ is surjective. Moreover, for a fixed $\bar{\Theta} \in \TT^N$ we define the map
\begin{equation*}
\iota_{\bar{\Theta}}:\RN \to \TT^N, \;\; \iota_{\bar{\Theta}}(t)= \bar{\Theta} + \iota(t).
\end{equation*}
Let $\mathcal{C}^k(\TT^N)$ be the space of functions $\bm{u}:\RN^N\to\RN$, that are $1$-periodic in each argument and have continuous derivatives up to the $k$-th order. We will call a function $u:\RN \to \RN$ \textit{quasi-periodic (with frequency $\omega$)}, if there is a function $\bm{u} \in \mathcal{C}^0(\TT^N)$ such that 
\begin{equation*}
u(t)=\bm{u}(\iota(t)) \;\; \text{for all} \;\; t\in \RN.
\end{equation*}
Now, consider a measure-preserving embedding $f:\mathcal{D}\subset \TT^N\times(0,\infty) \to \TT^N\times(0,\infty)$, which has the special structure
\begin{equation} \label{embedding form}
f(\bar{\Theta},r)=(\bar{\Theta}+\iota(F(\bar{\Theta},r)),r + G(\bar{\Theta},r)),
\end{equation}
where $F,G:\mathcal{D}\to \RN$ are continuous. For $\bar{\Theta} \in \TT^N$ let
\begin{equation*}
{D}_{\bar{\Theta}} = (\iota_{\bar{\Theta}} \times \id)^{-1}(\mathcal{D}) \subset \RN \times (0,\infty).
\end{equation*}
On these open sets we define the maps $f_{\bar{\Theta}}: {D}_{\bar{\Theta}} \subset \RN \times (0,\infty) \to  \RN \times (0,\infty)$ given by
\begin{equation} \label{planar maps}
f_{\bar{\Theta}}(t,r)=(t+F(\bar{\Theta}+\iota(t),r),r+G(\bar{\Theta}+\iota(t),r)).
\end{equation}
Then $f_{\bar{\Theta}}$ is continuous and meets the identity
\begin{equation*}
f \circ (\iota_{\bar{\Theta}} \times \id) = (\iota_{\bar{\Theta}} \times \id) \circ f_{\bar{\Theta}} \;\; \text{on} \;\; D_{\bar{\Theta}},
\end{equation*} 
i.e. the following diagram is commutative:
\begin{equation*}
\begin{tikzcd}
\mathcal{D} \arrow[r, "f"] 
& \TT^N\times(0,\infty)  \\
{D}_{\bar{\Theta}} \arrow[r, "f_{\bar{\Theta}}"] \arrow[u, "\iota_{\bar{\Theta}} \times \id"]
& \RN \times (0,\infty) \arrow[u, "\iota_{\bar{\Theta}} \times \id"]
\end{tikzcd}
\end{equation*}
Therefore $f_{\bar{\Theta}}$ is injective as well. Again we define $D_{\bar{\Theta},1} = D_{\bar{\Theta}}$ and $D_{\bar{\Theta},n+1} = f_{\bar{\Theta}}^{-1}(D_{\bar{\Theta},n})$ to construct the set
\begin{equation*}
D_{\bar{\Theta}, \infty} = \bigcap\limits_{n=1}^{\infty} D_{\bar{\Theta},n} \subset \RN \times (0,\infty),
\end{equation*}
where the forward iterates $(t_n,r_n)=f_{\bar{\Theta}}^n(t_0,t_0)$ are defined for all $n\in\NN$. This set is equivalently defined through the relation
\begin{align*}
D_{\bar{\Theta}, \infty} =  (\iota_{\bar{\Theta}} \times \id)^{-1}(\mathcal{D}_\infty).
\end{align*}
Now we can define the \textit{escaping set}
\begin{equation*}
{E_{\bar{\Theta}}}= \{(t_0,r_0)\in{D}_{{\bar{\Theta}},\infty}: \lim_{n\to\infty}r_n=\infty  \}.
\end{equation*}
Finally we are in position to state the theorem \cite[Theorem~3.1]{kunze_ortega_ping_pong}:
\begin{thm} \label{thm escaping set nullmenge}
	Let $f:\mathcal{D}\subset \TT^N \times (0,\infty) \to \TT^N \times (0,\infty)$ be a measure-preserving embedding of the form (\ref{embedding form}) and suppose that there is a function $W=W(\bar{\Theta},r)$ satisfying $W\in \mathcal{C}^1(\TT^N \times (0,\infty))$,
	\begin{equation*}
	0<\beta \leq \p_r W(\bar{\Theta},r) \leq \delta \;\; \text{for} \;\; \bar{\Theta} \in \TT^N, \;\; r \in (0,\infty),
	\end{equation*}
	with some constants $\beta,\delta>0$, and furthermore
	\begin{equation} \label{ineq adiabatic inv}
	W(f(\bar{\Theta},r)) \leq W(\bar{\Theta},r) + k(r) \;\; \text{for} \;\; (\bar{\Theta},r)\in \mathcal{D},
	\end{equation}
	where $k:(0,\infty) \to \RN$ is a decreasing and bounded function such that $\lim_{r\to\infty} k(r)=0$. Then, for allmost all $\bar\Theta \in \TT^N$, the set $E_{\bar{\Theta}}\subset \RN \times (0,\infty)$ has Lebesgue measure zero.
\end{thm}

The function $W$ can be seen as a generalized adiabatic invariant, since any growth will be slow for large energies.

\section{Transformation of the problem} \label{section transformations}
For $\alpha \geq 3$, consider the second order differential equation 
\begin{equation} \label{DGL 1}
\ddot{x} + \lvert x \rvert^{\alpha-1}x = p(t),
\end{equation}
where $p\in \mathcal{C}_b^4(\RN)$ is a (in general) non periodic forcing function. Here, $\mathcal{C}_b^k(\RN)$ denotes the space of bounded functions with continuous and bounded derivatives up to order $k$.
Solutions $x(t)$ to (\ref{DGL 1}) are unique and exist for all time. To see this, we define
\begin{equation*}
E(t)= \frac{1}{2}\dot{x}(t)^2 + \frac{1}{\alpha +1}\lvert x(t) \rvert^{\alpha+1}.
\end{equation*}
Then $\lvert \dot{E} \rvert = \lvert p(t) \dot{x} \rvert \leq \lvert p(t) \rvert \sqrt{2E}$, and therefore
\begin{equation*}
\sqrt{E(t)} \leq \sqrt{E(t_0)} + \frac{1}{\sqrt{2}} \left\lvert \int_{t_0}^{t} \lvert p(s)\rvert ds \right\rvert.
\end{equation*}
So $E$ is bounded on finite intervals and thus $x$ can be continued on $\RN$.

\subsection{Action-angle coordinates}

First we want to reformulate (\ref{DGL 1}) in terms of the action-angle coordinates of the unperturbed system, that is
\begin{equation} \label{DGL unperturbed}
\ddot{x} + \lvert x \rvert^{\alpha-1}x = 0.
\end{equation}
The orbits of (\ref{DGL unperturbed}) are closed curves, defined by $\frac{1}{2}y^2 + \frac{1}{\alpha +1}\lvert x \rvert^{\alpha+1}=\text{const.}$ and correspond to periodic solutions. For $\lambda>0$ let $x_\lambda$ denote the solution of (\ref{DGL unperturbed}) having the initial values
\begin{equation*}
x_\lambda(0)=\lambda, \;\; \dot x_\lambda(0)=0.
\end{equation*}
Using the homogeneity of the problem, we get $x_\lambda(t)=\lambda x_1(\lambda^\frac{\alpha-1}{2} t)$. In particular $x_\lambda$ has a decreasing minimal period $T(\lambda) = \lambda^\frac{1-\alpha}{2} T(1)$. Thus we can find the unique number $\Lambda>0$, such that $T(\Lambda)=2\pi$. We will use the notation 
\begin{equation*}
c(t) = x_{\Lambda}(t), \;\; s(t)=\dot c (t),
\end{equation*}
since in a lot of ways these functions behave like the trigonometric functions $\cos$ and $\sin$: $c$ is even, $s$ is odd, and both are anti-periodic with period $\pi$. Hence they have zero mean value, i.e.
\begin{equation*}
\int_{0}^{2\pi} c(t) \, dt = \int_{0}^{2\pi} s(t) \, dt = 0.
\end{equation*}
In this case however, $(c(t),s(t))$ spins clockwise around the origin of the $(x,\dot{x})$-plane. Furthermore $c$ and $s$ meet the identity
\begin{equation} \label{Identi}
\frac{1}{2}s(t)^2 +\frac{1}{\alpha +1} \lvert c(t) \rvert^{\alpha+1} = \frac{1}{\alpha + 1} \Lambda^{\alpha + 1} \;\; \forall t \in \RN.
\end{equation}
Now we can define a change of variables $\eta: \mathbb{S}^1\times(0,\infty) \to \RN^2\setminus\{0\}, \;  (\bar{\vartheta},r)\mapsto (x,v)$ by 
\begin{equation*} 
x = \gamma r^\frac{2}{\alpha+3} c(\bar{\vartheta}), \;\; v = \gamma^\frac{\alpha+1}{2} r^\frac{\alpha+1}{\alpha +3} s(\bar{\vartheta}),
\end{equation*}
where $\gamma>0$ is determined by
\begin{equation*}
\gamma^\frac{\alpha+3}{2}\frac{2}{\alpha +3} \Lambda^{\alpha +1} = 1.
\end{equation*}
This choice of $\gamma$ makes $\eta$ a symplectic diffeomorphism, as can be shown by an easy calculation.
Moreover, from (\ref{Identi}) follows the identity 
\begin{equation} \label{eq energy action}
\frac{1}{2}v^2  +\frac{1}{\alpha +1} \lvert x \rvert^{\alpha+1} = \kappa_1 r^\frac{2(\alpha+1)}{\alpha+3},
\end{equation}
where $\kappa_1= \frac{1}{\alpha+1}(\gamma\Lambda)^{\alpha+1}$. Adding a new component for the time, we define the transformation map  
\begin{equation*}
\mathcal{R}:\RN^2\setminus\{0\}\times \RN \to \mathbb{S}^1\times(0,\infty) \times \RN, \;\; \mathcal{R}(x,v;t)=(\eta^{-1}(x,v);t).
\end{equation*}
Going back to the perturbed system, the old Hamiltonian 
\begin{equation*}
h (x,\dot x;t) = \frac{1}{2}\dot{x}^2  +\frac{1}{\alpha +1} \lvert x \rvert^{\alpha+1} - p(t)x
\end{equation*}
expressed in the new coordinates is
\begin{equation}
\mathcal{H}(\bar{\vartheta},r ;t) = \kappa_1 r^\frac{2(\alpha+1)}{\alpha+3} - \gamma r^\frac{2}{\alpha+3} p(t) c(\bar{\vartheta}).
\end{equation}
For simplicity's sake let us denote the lift of $\mathcal{H}$ onto $\RN \times (0,\infty) \times \RN$ by the same letter $\mathcal{H}$. The associated differential equations then become
\begin{equation} \label{DGL 3}
\begin{cases}
\dot{{\vartheta}} &= \p_r \mathcal{H} = \frac{2(\alpha+1)}{\alpha+3}\kappa_1 r^\frac{\alpha-1}{\alpha+3} - \frac{2}{\alpha+3}\gamma r^{-\frac{\alpha+1}{\alpha+3}} p(t) c({\vartheta}) \\
\dot r  &= - \p_{{\vartheta}} \mathcal{H} = \gamma r^\frac{2}{\alpha+3} p(t) s({\vartheta})
\end{cases}.
\end{equation}
It should be noted that solutions to (\ref{DGL 3}) only exist on intervals $J \subset \RN$, where $r(t)>0$. Therefore, we can only make assertions about solutions of the original problem (\ref{DGL 1}) defined on intervals, where $(x,\dot{x}) \neq 0$, when working with these action-angle coordinates.

\subsection{Time-energy coordinates}
In order to construct a measure preserving embedding one could take the Poincaré map of Hamiltonian system (\ref{DGL 3}). However, to fit the setting of subsection \ref{subsection quasi-periodic functions}, this map would need to have the time (and thus the quasi-periodic dependence of the system) as the first variable. Therefore, we will follow \cite{levi_quasiperiodic_motions} and take the time $t$ as the new ``position''-coordinate, the energy $\mathcal{H}$ as the new ``momentum'' and the angle $\vartheta$ as the new independent variable. 
\ \\
Since the first term in (\ref{DGL 3}) is dominant for $r\to \infty$ one can find $r_*$ such that
\begin{equation} \label{Mono abs}
\p_r \mathcal{H}(\vartheta,r;t) \geq 1 \;\; \text{for all} \;\; r\geq r_*.
\end{equation}
\begin{rmk} \label{remark constants}
	The value of $r_*$ depends upon $\alpha,\gamma, \kappa_1,\lVert c \rVert_{\mathcal{C}_b}$ and $\lVert p\rVert_{\mathcal{C}^4_b}$, where again $\gamma, \kappa_1,\lVert c \rVert_{\mathcal{C}_b}$ are uniquely determined by the choice of $\alpha$. We will call quantities depending only upon $\alpha$ and $\lVert p\rVert_{\mathcal{C}^4_b}$ constants. Let us also point out, that $r_*$ can be chosen ``increasingly in $\lVert p\rVert_{\mathcal{C}^4_b}$''. By this we mean, that if $r_*=r_*(\alpha,\lVert p\rVert_{\mathcal{C}^4_b})$ is the threshold corresponding to some $p \in \mathcal{C}^4_b(\RN)$, then (\ref*{Mono abs}) also holds for any forcing $\tilde{p}\in \mathcal{C}^4_b$ with $\lVert \tilde{p} \rVert_{\mathcal{C}^4_b} \leq \lVert p\rVert_{\mathcal{C}^4_b}$. Indeed, all thresholds we will construct have this property.
\end{rmk}
Now consider a solution $(\vartheta,r)$ of (\ref{DGL 3}) defined on an interval $J$, where $r(t)>r_*$ for all $t \in J$. Than the function $t\mapsto \vartheta(t)$ is invertible, since
\begin{equation*}
\dot\vartheta (t) = \p_r \mathcal{H}(\vartheta(t),r(t);t) \geq 1.
\end{equation*}
Adopting the notation of \cite{kunze_ortega_stability_estimates}, we will write $\tau=\vartheta(t)$ and denote the inverse by $\phi$, i.e. $\phi(\tau)=t$. Since $\vartheta(t)$ is at least of class $\mathcal{C}^2$, the same holds for the inverse function $\phi$ defined on $\vartheta(J)$. Let us now define 
\begin{equation*}
I(\tau)=\mathcal{H}(\tau,r(\phi(\tau));\phi(\tau)) \;\; \text{for} \; \; \tau \in \vartheta(J).
\end{equation*}
This function will be the new momentum. 
It is a well known fact that the resulting system is again Hamiltonian. To find the corresponding Hamiltonian, we can solve the equation
\begin{equation*}
\mathcal{H}(\vartheta,H;t)=I
\end{equation*}
implicitly for $H(t,I;\vartheta)$. Because of (\ref{Mono abs}) this equation admits a solution, which is well-defined on the open set
\begin{equation*}
\Omega=\{(t,I;\vartheta)\in \RN^3 : I > \mathcal{H}(\vartheta, r_* ;t)\}.
\end{equation*}
Indeed, by implicit differentiation it can be verified that
\begin{equation*}
\phi'=\p_I H, \;\; I'=-\p_\phi H,
\end{equation*}
where the prime $'$ indicates differentiation with respect to $\tau$.
Using the new coordinates, we have to solve
\begin{equation} \label{Impl Ham 1}
\kappa_1 H^\frac{2(\alpha+1)}{\alpha+3} - \gamma H^\frac{2}{\alpha+3} p(\phi) c(\tau) = I
\end{equation}
or equivalently
\begin{equation} \label{Impl Ham 2}
H = I^\frac{\alpha+3}{2(\alpha+1)}\kappa_1^{-\frac{\alpha+3}{2(\alpha+1)}} (1- \kappa_1^{-1} \gamma H^\frac{-2\alpha}{\alpha+3} p(\phi) c(\tau))^{-\frac{\alpha+3}{2(\alpha+1)}}.
\end{equation}
Since $p\in \mathcal{C}^6$ and $c\in \mathcal{C}^3$, also $H$ will be of class $\mathcal{C}^3$. Moreover, we can find $I_*>0$ (depending upon  $\alpha,\gamma, \kappa_1,\lVert c \rVert_{\mathcal{C}_b},\lVert p\rVert_{\mathcal{C}_b}$ and $r_*$) such that
\begin{equation*}
\{(\phi,I;\tau)\in \RN^3 : I\geq I_* \} \subset \Omega.
\end{equation*}
Furthermore, we can choose $I_*$ so large that the solution $H$ of (\ref{Impl Ham 1}) satisfies
\begin{equation*} 
\alpha_0  I^\frac{\alpha+3}{2(\alpha+1)} \leq H \leq \beta_0 I^\frac{\alpha+3}{2(\alpha+1)} \;\; \text{for} \;\; I\geq I_*
\end{equation*}
for some constants $\alpha_0,\beta_0>0$.
Let 
\begin{equation*}
\kappa_0=\kappa_1^{-\frac{\alpha+3}{2(\alpha+1)}}=\left(\frac{2(\alpha+1)}{\alpha+3}\gamma^\frac{1-\alpha}{2}\right)^\frac{\alpha+3}{2(\alpha+1)}.
\end{equation*}
To approximate the solution $H(\phi,I;\tau)$ of (\ref{Impl Ham 2}), one can use the Taylor polynomial of degree one for $(1-z)^{-\frac{\alpha+3}{2(\alpha+1)}}$ and then plug in the highest order approximation $\kappa_0 I^\frac{\alpha+3}{2(\alpha+1)}$ for the remaining $H$ on the right-hand side. Therefore we define the remainder function $R \in \mathcal{C}^3(G)$ through the relation
\begin{equation} \label{Ham 3}
H(\phi,I;\tau)=\kappa_0 I^\frac{\alpha+3}{2(\alpha+1)} + \frac{(\alpha+3)}{2(\alpha+1)}\gamma \kappa_0^\frac{\alpha+5}{\alpha + 3}  p(\phi) c(\tau) I^\frac{3-\alpha}{2(\alpha+1)} + R(\phi,I;\tau).
\end{equation}
The corresponding system is described by
\begin{equation} 
\begin{cases} \label{Ham 3 equations}
\phi' &= \p_I H = \kappa_0 \frac{\alpha+3}{2(\alpha+1)} I^\frac{1-\alpha}{2(\alpha+1)} + \frac{9-\alpha^2}{4(\alpha+1)^2}\gamma \kappa_0^\frac{\alpha+5}{\alpha+3} p(\phi)c(\tau) I^\frac{1-3\alpha}{2(\alpha+1)} + \p_I R,   \\
I' &= -\p_\phi H = -\frac{\alpha+3}{2(\alpha+1)} \gamma \kappa_0^\frac{\alpha+5}{\alpha+3} \dot p(\phi)c(\tau) I^\frac{3-\alpha}{2(\alpha+1)} - \p_\phi R.
\end{cases}
\end{equation}
The change of variables $(\vartheta,r;t)\mapsto(\phi,I;\tau)$ can be realized via the transformation map $\mathcal{S}:\RN \times [r_*,\infty) \times \RN \to \RN \times (0,\infty) \times \RN$ defined by
\begin{equation*}
\mathcal{S}(\vartheta,r;t) = (t,\mathcal{H}(\vartheta,r;t);\vartheta).
\end{equation*}
So $\mathcal{S}$ maps a solution $(\vartheta(t),r(t))$ of (\ref{DGL 3}) with the initial condition $(\vartheta(t_0),r(t_0))=(\vartheta_0,r_0)$ onto a solution $(\phi(\tau),I(\tau))$ of (\ref{Ham 3 equations}) with initial condition $(\phi(\vartheta_0),I(\vartheta_0))=(t_0,\mathcal{H}(\vartheta_0,r_0;t_0))$. \\
The following lemma by Kunze and Ortega \cite[Lemma~7.1]{kunze_ortega_stability_estimates} shows that $R$ is small in a suitable sense:
\begin{lem} \label{lem remainder}
 There are constants $C_0>0$ and $I_{C_0} \geq I_*>0$ (depending upon $\lVert p \rVert_{\mathcal{C}_b^2(\RN)}$) such that
 \begin{equation} \label{Remainder Abs}
 \lvert R \rvert + \lvert \p_\phi R \rvert + I\lvert \p_I R \rvert + \lvert \p^2_{\phi\phi} R \rvert + I\lvert \p^2_{\phi I} R \rvert + I^2\lvert \p^2_{II} R \rvert \leq C_0 I^\frac{3(1-\alpha)}{2(\alpha+1)}
 \end{equation}
 holds for all $\phi,\tau \in \RN$ and $I\geq I_{C_0}$.
\end{lem}
Now we could use these coordinates and a corresponding Poincaré map for Theorem \ref{thm escaping set nullmenge}. But since it can be tricky to find a suitable function $W$, we would like the energy variable $I(\tau)$ itself to be an adiabatic invariant in the sense of (\ref{ineq adiabatic inv}) . However, for $\alpha = 3$ we do not have $I'\to0$ as $I\to \infty$. Therefore we have to do one further transformation. For $\alpha > 3$ this last step would not be necessary.

\subsection{A last transformation}

In \cite[Theorem~6.7]{kunze_ortega_stability_estimates} Ortega and Kunze constructed a change of coordinates, which reduces the power of the momentum variable in the second term of (\ref{Ham 3}) while preserving the special structure of the Hamiltonian. Since in their paper they had to use this transformation several times consecutively, the associated theorem is somewhat general and too complicated for our purpose here. Thus we will cite it only in the here needed form. \\
For $\mu>0$ we set
\begin{equation*}
\Sigma_{\mu}=\RN\times [\mu,\infty) \times \RN.
\end{equation*}

\begin{thm} \label{thm ham transformation}
	Consider the Hamiltonian $H$ from (\ref{Ham 3}), i.e.
	\begin{equation*}
	H(\phi,I;\tau)=\kappa_0 I^\frac{\alpha+3}{2(\alpha+1)} +  f(\phi) c(\tau)  I^\frac{3-\alpha}{2(\alpha+1)} + R(\phi,I;\tau),
	\end{equation*} 
	where $f(\phi)= \frac{(\alpha+3)}{2(\alpha+1)}\gamma \kappa_0^\frac{\alpha+5}{\alpha + 3} p(\phi)$, and $I_{C_0}$ from Lemma \ref{lem remainder}. Then there exists $I_{**}>I_{C_0}$, $\mathcal{I}_*>0$ and a $\mathcal{C}^1$-diffeomorphism 
	\begin{equation*}
	\mathcal{T}:\Sigma_{I_{**}} \to \mathcal{T}(\Sigma_{I_{**}} ) \subset \Sigma_{\mathcal{I}_*} , \;\;  (\phi,I;\tau)\mapsto (\varphi,\mathcal{I};\tau),
	\end{equation*}
	which transforms the system (\ref{Ham 3 equations}) into $\varphi'=\p_\mathcal{I}H_1, \mathcal{I}'=-\p_\varphi H_1$, where
	\begin{equation*}
	H_1(\varphi,\mathcal{I};\tau) = \kappa_0 \mathcal{I}^\frac{\alpha+3}{2(\alpha+1)} + f_1(\varphi)c_1(\tau)\mathcal{I}^{b_\alpha} + R_1(\varphi,\mathcal{I};\tau).
	\end{equation*}
	The new functions appearing in $H_1$ satisfy
	\begin{enumerate} [label=(\alph*)]
		\item $f_1(\varphi)= -\frac{\alpha+3}{2(\alpha+1)}\kappa_0 \dot f(\varphi) =   -\left(\frac{\alpha+3}{2(\alpha+1)}\right)^2 \gamma  \kappa_0^\frac{2\alpha+8}{\alpha+3} \dot p (\varphi)$,
		\item $c_1 \in \mathcal{C}^4(\RN), c_1'(\tau)=c(\tau), \int_{0}^{2\pi} c_1(\tau) \, d\tau=0,$
		\item $b_\alpha= -\frac{3\alpha^2 - 2 \alpha - 9}{2(\alpha +3)(\alpha+1)} < \frac{3-\alpha}{2(\alpha+1)} \leq 0$, and
		\item $R_1 \in \mathcal{C}^{3}(\Sigma_{\mathcal{I}_{*}})$ satisfies (\ref{Remainder Abs}) for all $\mathcal{I}\geq \mathcal{I}_*$ and with some constant $\tilde{C}_0>0$.
	\end{enumerate}
	The quantities $I_{**},\mathcal{I_{*}}$ and $\tilde{C}_0$ can be estimated in terms of $\alpha,\kappa_0, \lVert f \rVert_{\mathcal{C}_b^4(\RN)} , \lVert c \rVert_{\mathcal{C}_b(\RN)}$, and $C_0$ from Lemma \ref{lem remainder}. Furthermore, the change of variables $\mathcal{T}$ has the following properties:
	\begin{enumerate}[label=(\roman*)]
		\item $\mathcal{T}(\cdot,\cdot;\tau)$ is symplectic for all $\tau \in \RN$, i.e. $d\varphi \wedge d\mathcal{I} = d\phi \wedge dI$,
		\item $\mathcal{T}(\phi,I;\tau+2\pi) = \mathcal{T}(\phi,I;\tau) + (0,0;2\pi)$, and
		\item $I/2 \leq \mathcal{I}(\phi,I;\tau) \leq 2I$ for all $(\phi,I;\tau)$.
	\end{enumerate}
\end{thm}
Even if we omit the proof here, let us note that the change of variables can be realized via the generating function
\begin{equation*}
\Psi(\phi,\mathcal{I}; \tau) = -\mathcal{I}^\frac{3-\alpha}{2(\alpha+1)} f(\phi) c_1(\tau),
\end{equation*}
where $c_1$ is uniquely determined by the conditions in $(b)$. Therefore $\mathcal{T}$ is implicitly defined by the equations
\begin{equation*}
I = \mathcal{I} + \p_\phi \Psi, \;\; \varphi = \phi + \p_\mathcal{I} \Psi
\end{equation*}
and one can determine $H_1$ through the relation
\begin{equation*}
H_1(\varphi,\mathcal{I};\tau) = H(\phi,I;\tau) + \p_{\tau} \Psi(\phi,\mathcal{I};\tau).
\end{equation*}

\section{The successor map} \label{section successor map}

Consider the new Hamiltonian from Theorem \ref{thm ham transformation}, that is
\begin{equation*}
	H_1(\varphi,\mathcal{I};\tau) = \kappa_0 \mathcal{I}^\frac{\alpha+3}{2(\alpha+1)} + f_1(\varphi)c_1(\tau)\mathcal{I}^{b_\alpha} + R_1(\varphi,\mathcal{I};\tau),
\end{equation*}
which is well-defined on the set $\mathcal{T}(\Sigma_{I_{**}})$ and $2\pi$-periodic in the time variable $\tau$. The corresponding equations of motion are
\begin{align} \label{Ham 4 equations}
\begin{cases}
\varphi' &= \p_{\mathcal{I}} H_1 =  \kappa_0\frac{\alpha+3}{2(\alpha+1)}\mathcal{I}^{\frac{1-\alpha}{2(\alpha+1)}} 
+b_\alpha f_1(\varphi)c_1(\tau)\mathcal{I}^{b_\alpha -1} + \p_\mathcal{I} R_1, \\
\mathcal{I}'&= -\p_\varphi H_1 = -\dot f_1 (\varphi)  c_1(\tau)\mathcal{I}^{b_\alpha} - \p_\varphi R_1,
\end{cases}
\end{align}
where $\dot f_1 (\varphi) = - \left(\frac{\alpha+3}{2(\alpha+1)}\right)^2 \kappa_0^\frac{2\alpha+8}{\alpha+3} \gamma \ddot{p} (\varphi)$.\\
Now suppose $(\varphi_0,\mathcal{I}_0;\tau_0)\in \mathcal{T}(\Sigma_{I_{**}})$ and denote by $(\varphi(\tau;\varphi_0,\mathcal{I}_0,\tau_0),\mathcal{I}(\tau;\varphi_0,\mathcal{I}_0,\tau_0))$ the solution of (\ref{Ham 4 equations}) with the initial data
\begin{equation*}
\varphi(\tau_0)= \varphi_0, \;\; \mathcal{I}(\tau_0)= \mathcal{I}_0.
\end{equation*}
We want to construct a subset $\Sigma_{\mathcal{I}_{**}} = \RN \times [\mathcal{I}_{**},\infty) \times \RN \subset \mathcal{T}(\Sigma_{I_{**}})$ such that $(\varphi,\mathcal{I})$ is defined on the whole interval $[\tau_0,\tau_0+2\pi]$ whenever $(\varphi_0,\mathcal{I}_0,\tau_0)\in \Sigma_{\mathcal{I}_{**}}$. Similar to \cite[Lemma~4.1]{kunze_ortega_stability_estimates}, we state:
\begin{lem} \label{lem I abs}
	There exists a constant $\mathcal{I_{**}}>  \mathcal{I}_*$ (depending only upon $\alpha, \lVert f \rVert_{\mathcal{C}_b^4(\RN)} , \lVert c \rVert_{\mathcal{C}_b(\RN)}$ and $\tilde{C}_0$ from Theorem \ref{thm ham transformation}) such that $\Sigma_{\mathcal{I}_{**}} \subset \mathcal{T}(\Sigma_{I_{**}})$ and for any $(\varphi_0,\mathcal{I}_0,\tau_0) \in \Sigma_{\mathcal{I}_{**}} $ the solution $(\varphi,\mathcal{I})$ of (\ref{Ham 4 equations}) with initial data
	\begin{equation*}
	\varphi(\tau_0)= \varphi_0, \;\; \mathcal{I}(\tau_0)= \mathcal{I}_0
	\end{equation*}
	exists on $[\tau_0,\tau_0+2\pi]$, where it satisfies
	\begin{equation} \label{abs I aus lemma existenzint}
	\frac{\mathcal{I}_0}{4} \leq \mathcal{I}(\tau) \leq 4 \mathcal{I}_0  \;\; \text{for} \;\; \tau \in [\tau_0,\tau_0+2\pi].
	\end{equation}
\end{lem}

\begin{proof}
	Suppose $\mathcal{I}_0 \geq \mathcal{I}_{**}\geq 4 \mathcal{I}_{*}$, then $(iii)$ from Theorem \ref{thm ham transformation} yields $\Sigma_{\mathcal{I}_{**}} \subset \mathcal{T}(\Sigma_{I_{**}})$. Now, let $T>0$ be maximal such that $\mathcal{I}_0/4 \leq \mathcal{I}(\tau) \leq 4 \mathcal{I}_0$ holds for all $\tau \in [\tau_0,\tau_0 + T)$. On this interval we have
	\begin{equation*}
	(\mathcal{I}^{1-b_\alpha})'=(1-b_\alpha)\mathcal{I}^{-b_\alpha}\mathcal{I}' = (1-b_\alpha)\mathcal{I}^{-b_\alpha}(-\dot f_1 (\varphi)  c_1(\tau)\mathcal{I}^{b_\alpha} - \p_\varphi R_1)
	\end{equation*}
	and thus
	\begin{equation*}
	\lvert (\mathcal{I}^{1-b_\alpha})' \rvert 
	\leq \lvert (1-b_\alpha)  \rvert \left( \lVert \dot f_1 \rVert_{\mathcal{C}_b} \; \lVert c_1 \rVert_{\mathcal{C}_b} +\lvert \p_\varphi R_1 \rvert \mathcal{I}^{-b_\alpha} \right)
	\leq \lvert (1-b_\alpha)  \rvert \left( \lVert \dot f_1 \rVert_{\mathcal{C}_b} \; \lVert c_1 \rVert_{\mathcal{C}_b} + \tilde{C}_0 \right) = \hat{C},
	\end{equation*}
	with $\tilde{C}_0>0$ from $(d)$ of Theorem \ref{thm ham transformation}, since $b_\alpha = -\frac{3\alpha^2 - 2 \alpha - 9}{2(\alpha +3)(\alpha+1)} > \frac{3(1-\alpha)}{2(\alpha+1)}$. Now assume $T\leq 2\pi$, then for $\mathcal{I}_{**}$ sufficiently large we conclude
	\begin{equation*}
	\left(\frac{\mathcal{I}_0}{2}\right)^{1-b_\alpha} \leq \mathcal{I}_0^{1-b_\alpha} - 2\pi \hat{C} \leq \mathcal{I}(\tau)^{1-b_\alpha} \leq  \mathcal{I}_0^{1-b_\alpha} + 2\pi \hat{C} \leq  (2\mathcal{I}_0)^{1-b_\alpha}
	\end{equation*}
	on the whole interval $[\tau_0,\tau_0 + T)$. This contradicts the definition of $T$ and thus completes the proof.
\end{proof}
We can therefore consider the Poincaré map $\Phi:\RN \times [\mathcal{I}_{**},\infty) \to \RN^2$ corresponding to the periodic system (\ref{Ham 4 equations}), defined by
\begin{equation} \label{Def successor}
\Phi(\varphi_0,\mathcal{I}_0) = (\varphi(5\pi/2;\varphi_0,\mathcal{I}_0,\pi/2),\mathcal{I}(5\pi/2;\varphi_0,\mathcal{I}_0,\pi/2)).
\end{equation}
The choice $\tau_0=\frac{\pi}{2}$ is basically due to computational advantages, since $c(\vartheta)=0$ if and only if $\vartheta = \pi/2 + m\pi$ with $m \in \IN$. Moreover, values of $\tau = \vartheta$ in $\pi/2 + 2\pi \IN$ correspond exactly to those zeros of the solution $x(t)$, where $\dot{x}<0$.
\begin{equation*}
\Phi(\varphi_0,\mathcal{I}_0) =  (\varphi_1,\mathcal{I}_1) .
\end{equation*}
Now that we have defined a suitable successor map, we can prove that $\mathcal{I}$ is an adiabatic invariant in the sense of equation (\ref{ineq adiabatic inv}):
\begin{lem} \label{lem adiabatic invariant}
	There is a constant $C>0$ (depending only upon $\alpha, \lVert f \rVert_{\mathcal{C}_b^4(\RN)} , \lVert c \rVert_{\mathcal{C}_b(\RN)}$  and $\tilde{C}_0$) such that 
	\begin{equation*}
	\lvert \mathcal{I}_1 - \mathcal{I}_0 \rvert \leq C \mathcal{I}_0^{b_\alpha}
	\end{equation*}
	holds for all $(\varphi_0,\mathcal{I}_0) \in \RN \times [\mathcal{I}_{**},\infty)$.
\end{lem}

\begin{proof}
	With a similar reasoning like in the proof of Lemma \ref{lem I abs} we get
	\begin{align*}
	\lvert \mathcal{I}' (\tau) \rvert &= \lvert -\dot f_1 (\varphi)  c_1(\tau)\mathcal{I}^{b_\alpha}(\tau) - \p_\varphi R_1 \rvert
	\leq  \lVert \dot f_1 \rVert_{\mathcal{C}_b} \; \lVert c_1 \rVert_{\mathcal{C}_b} \mathcal{I}(\tau)^{b_\alpha} + \tilde{C}_0  \mathcal{I}(\tau)^{ \frac{3(1-\alpha)}{2(\alpha+1)}} \\
	&\leq  \left( \lVert \dot f_1 \rVert_{\mathcal{C}_b} \; \lVert c_1 \rVert_{\mathcal{C}_b}  + \tilde{C}_0 \right)\mathcal{I}(\tau)^{b_\alpha} 
	\leq \left( \lVert \dot f_1 \rVert_{\mathcal{C}_b} \; \lVert c_1 \rVert_{\mathcal{C}_b}  + \tilde{C}_0 \right) 4^{-b_\alpha}  \mathcal{I}_0^{b_\alpha}.
	\end{align*}
	Now integrating over $[\pi/2,5\pi/2]$ gives us
	\begin{equation*}
		\lvert \mathcal{I}_1 - \mathcal{I}_0 \rvert \leq 2\pi \left( \lVert \dot f_1 \rVert_{\mathcal{C}_b} \; \lVert c_1 \rVert_{\mathcal{C}_b}  + \tilde{C}_0 \right) 4^{-b_\alpha}  \mathcal{I}_0^{b_\alpha}.
	\end{equation*}
\end{proof}

\section{Quasi-periodicity} \label{section quasi periodicity}

So far all our considerations have dealt with the case of a general forcing function $p \in \mathcal{C}^{4}_b(\RN)$. Now we will replace $p(t)$ by $p_{\bar{\Theta}}(t)$ from (\ref{def quasiperiodic p}) and show that the quasi-periodicity is inherited by the Hamiltonian system (\ref{Ham 4 equations}). But first let us clarify some notation. \\
In this section, we will mark continuous functions with an argument in $\TT^N$ with bold letters. Each such function $\bm{u}$ gives rise to a family of quasi-periodic maps $\{u_{\bar{\Theta}}\}_{\bar{\Theta}\in \TT^N}$ via the relation
\begin{equation*}
u_{\bar{\Theta}}(t)=\bm{u}(\iota_{\bar{\Theta}}(t)).
\end{equation*}
Note that since $\iota_{\bar{\Theta}}(\RN)$ lies dense in $\TT^N$, the function $\bm{u}\in \mathcal{C}(\TT^N)$ is also uniquely determined by this property. Moreover, for $\bm{u}\in \mathcal{C}^1(\TT^N)$ we introduce the notation $\p_\omega = \sum_{i=1}^{N} \omega_i \frac{\p}{\p \theta_i}$, so that
\begin{equation*}
\frac{d}{d t} u_{\bar{\Theta}}(t)= \p_\omega \bm{u}(\iota_{\bar{\Theta}}(t)).
\end{equation*}
So let us plug the forcing $p_{\bar{\Theta}}(t)$ associated to $\bm{p}$ from the main theorem into (\ref{DGL 1}). Then $p_{\bar{\Theta}} \in \mathcal{C}^4_b(\RN)$, because
\begin{equation*}
\lVert p_{\bar{\Theta}} \rVert_{\mathcal{C}_b^4(\RN)} \leq \max(1, \lVert \omega \rVert^4_\infty)  \lVert \bm{p} \rVert_{\mathcal{C}^4(\TT^N)}
\end{equation*}
holds. Therefore all results of section \ref{section transformations} are applicable. Considering Remark \ref{remark constants}, this also implies that we can find new constants $r_*, I_*$ etc.\@ depending only upon $\alpha, \omega$ and $\lVert \bm{p} \rVert_{\mathcal{C}^4(\TT^N)}$ such that corresponding estimates hold uniformly in $\bar{\Theta} \in \TT^N$.\\
Since $\mathcal{R}$ and $\mathcal{S}$ basically leave the time variable $t$ unchanged, it is straightforward to prove that the transformation to action-angle coordinates as well as the change to the time-energy coordinates $(\phi,I)$, preserves the quasi-periodic structure (cf. \cite[p.~1242]{levi_zehnder_quasiperiodic}). Thus we find functions 
\begin{equation*}
\bm{H},\bm{R}:\TT^N \times [I_*,\infty) \times \RN \to \RN \times [I_*,\infty) \times \RN
\end{equation*}
of class $\mathcal{C}^{3}$ depending on $\bm{p}$ such that 
\begin{equation*}
H(\phi,I;\tau) = \bm{H}(\iota_{\bar{\Theta}}(\phi),I;\tau), \;\; R(\phi,I;\tau) = \bm{R}(\iota_{\bar{\Theta}}(\phi),I;\tau).
\end{equation*}
holds for every $\bar{\Theta}\in \TT^N$.
\begin{rmk}
	Let us note, that the functions $H,R$ etc.\@ now depend on the choice of ${\bar{\Theta}}$. Thus it would be more precise to write $H_{\bar{\Theta}},R_{\bar{\Theta}}$ and so on, but for reasons of clarity we will omit the index throughout this section. The functions $\bm{H},\bm{R}$ etc.\@ on the other hand are uniquely determined by $\bm{p}$.
\end{rmk}
However it requires a bit more work, to see that also the transformation $\mathcal{T}$ defined in Theorem \ref{thm ham transformation} retains the quasi-periodic properties. We recall that this change of variables is defined by
\begin{equation*}
I = \mathcal{I} + \p_\phi \Psi, \;\; \varphi = \phi + \p_\mathcal{I} \Psi,
\end{equation*}
where $\Psi(\phi,\mathcal{I};\tau) = - \mathcal{I}^\frac{3-\alpha}{2(\alpha+1)}f(\phi)c_1(\tau)$.
Let $\bm{f} \in \mathcal{C}^4(\TT^N)$ be the function one obtains by replacing $p$ in the definition of $f$ by $\bm{p}$.
Inspired by \cite[p.~261]{Moser_siegel_celmechanics}, we can prove the following lemma:
\begin{lem} \label{lem q quasip}
	Suppose $\mathcal{I}_{**} \geq  2\pi\max (1, \lVert \omega \rVert^4_\infty)\lVert \bm{f} \rVert_{\mathcal{C}^4(\TT^N)} \lVert c \rVert_{\mathcal{C}_b(\RN)}$ and consider
	\begin{equation*}
	\varphi(\phi,\mathcal{I};\tau) = \phi - \frac{3-\alpha}{2(\alpha+1)} \mathcal{I}^{\frac{1-3\alpha}{2(\alpha+1)}} c_1(\tau) f(\phi),
	\end{equation*}
	then for $\mathcal{I}\geq \mathcal{I}_{**},\tau \in \RN$ the inverse can be written in the form
	\begin{equation*}
	\phi(\varphi,\mathcal{I};\tau) = \varphi + q(\varphi,\mathcal{I};\tau),
	\end{equation*}
	where $q(\varphi,\mathcal{I};\tau) = \bm{q}(\iota_{\bar{\Theta}}(\varphi),\mathcal{I};\tau)$ with $\bm{q}\in \mathcal{C}^{3}(\TT^N\times [\mathcal{I}_{**},\infty)\times \RN)$.
\end{lem}

\begin{proof}
	First consider the function $w(\phi,\varphi,\mathcal{I},\tau)=\phi - \frac{3-\alpha}{2(\alpha+1)} \mathcal{I}^{\frac{1-3\alpha}{2(\alpha+1)}} c_1(\tau) f(\phi) - \varphi$, where $\mathcal{I}\geq \mathcal{I}_{**}$. Since $\mathcal{I}_{**} \geq \lVert f \rVert_{\mathcal{C}_b^4(\RN)}  \lVert c_1 \rVert_{\mathcal{C}_b(\RN)}$ we have $\p_\phi w \geq  \frac{1}{(\alpha+1)}$. Since also $w \in \mathcal{C}^{4}(\RN^2\times [\mathcal{I}_{**},\infty)\times \RN)$ the equation $w=0$ defines a unique solution $\phi(\varphi,\mathcal{I};\tau)$ of class $\mathcal{C}^{4}$.\\	
	If there was such a function $q$ we should have
	\begin{equation} \label{gleichung def q}
	q(\varphi,\mathcal{I};\tau)  = \phi -\varphi = \frac{3-\alpha}{2(\alpha+1)} \mathcal{I}^{\frac{1-3\alpha}{2(\alpha+1)}} c_1(\tau) f(\varphi + q(\varphi,\mathcal{I};\tau)).
	\end{equation}
	Let us now fix some $\mathcal{I}\geq \mathcal{I}_{**}$ and $\tau \in \RN$. For simplicity's sake we will drop the dependence of this variables and thus write $q(\varphi)$ for $q(\varphi,\mathcal{I};\tau)$ etc. Since $f(\phi)=\bm{f}(\iota_{\bar{\Theta}}(\phi))$, equation (\ref{gleichung def q}) translates into
	\begin{equation*}
	\bm{q}(\bar{\Omega}) 	- \frac{3-\alpha}{2(\alpha+1)} \mathcal{I}^{\frac{1-3\alpha}{2(\alpha+1)}} c_1(\tau) \bm{f}(\bar{\Omega} + \iota(\bm{q}(\bar{\Omega})))  = 0 \;\; \text{for all} \;\; \bar{\Omega}\in \TT^N.
	\end{equation*}
	We instead consider the equation
	\begin{equation*}
	\bm{q} 	-  \sigma \left( \frac{3-\alpha}{2(\alpha+1)} \mathcal{I}^{\frac{1-3\alpha}{2(\alpha+1)}} c_1(\tau) \bm{f}(\bar{\Omega} + \iota(\bm{q}))  \right)= 0, \;\; \sigma\in[0,1] 
	\end{equation*}
	and search for a solution $\bm{q}(\bar{\Omega}; \sigma)$. Differentiation with respect to $\sigma$ yields the ordinary differential equation
	\begin{equation*}
	\frac{d \bm{q}}{d \sigma} = \chi (\bar{\Omega}+\iota(\bm{q}); \sigma), \;\; \bm{q}(\bar{\Omega};0)=0,
	\end{equation*}
	where
	\begin{equation*}
	\chi (\bar{\Omega}; \sigma) = \frac{\frac{3-\alpha}{2(\alpha+1)} \mathcal{I}^{\frac{1-3\alpha}{2(\alpha+1)}}  c_1(\tau) \bm{f}(\bar{\Omega})}{1 - \sigma \frac{3-\alpha}{2(\alpha+1)} \mathcal{I}^{\frac{1-3\alpha}{2(\alpha+1)}}  c_1(\tau) \p_\omega\bm{f}(\bar{\Omega}) }.
	\end{equation*}
	As above, the denominator is greater than $1-\frac{3-\alpha}{2(\alpha+1)}>0$ by assumption. Thus, the right-hand side $\chi$ is $\mathcal{C}^{3}(\TT^N\times [0,1])$ for all $\bar{\Omega} \in \TT^N$. Therefore the differential equation has a unique solution $\bm{q}(\bar{\Omega},\sigma)$ for all $\sigma \in [0,1]$. Now, the function $q(\varphi)=\bm{q}(\iota_{\bar{\Theta}}(\varphi),1)$ has the desired properties.
\end{proof}
\begin{rmk}
	Without loss of generality we can impose the assumption of Lemma \ref{lem q quasip}, since $\lVert \bm{f} \rVert_{\mathcal{C}^4(\TT^N)} \leq \gamma \kappa_0^2 \lVert \bm{p} \rVert_{\mathcal{C}^4(\TT^N)}$ and thus $\mathcal{I}_{**}$ still depends only upon $\alpha, \omega$ and $\lVert \bm{p} \rVert_{\mathcal{C}^4(\TT^N)}$.
\end{rmk}
Now, since
\begin{equation*}
H_1(\varphi,\mathcal{I};\tau) = H(\phi,I;\tau) + \p_{\tau} \Psi(\phi,\mathcal{I};\tau),
\end{equation*}
where $ \p_{\tau} \Psi(\phi,\mathcal{I};\tau) = - \mathcal{I}^\frac{3-\alpha}{2(\alpha+1)}f(\phi)c(\tau)$, and because of Lemma \ref{lem q quasip} we have
\begin{equation*}
H_1(\varphi,\mathcal{I};\tau) = H(\varphi + q(\varphi,\mathcal{I};\tau),I;\tau)  - \mathcal{I}^\frac{3-\alpha}{2(\alpha+1)}f(\varphi + q(\varphi,\mathcal{I};\tau))c(\tau).
\end{equation*}
Moreover, $I$ can be expressed as
\begin{equation*}
I =\mathcal{I}  - \mathcal{I}^\frac{3-\alpha}{2(\alpha+1)}\dot f(\phi)c_1(\tau)
= \mathcal{I}  - \mathcal{I}^\frac{3-\alpha}{2(\alpha+1)}\dot f(\varphi + q(\varphi,\mathcal{I};\tau))c_1(\tau).
\end{equation*}
Hereby motivated, we define the $\mathcal{C}^3$-maps $\bm{I},\bm{H_1}:\TT^N \times [\mathcal{I}_{**},\infty)\times \RN \to \RN$ by
\begin{equation*}
\bm{I}(\bar{\Omega},\mathcal{I};\tau) = \mathcal{I}  - \mathcal{I}^\frac{3-\alpha}{2(\alpha+1)}\p_\omega \bm{f}(\bar{\Omega} + \iota(\bm{q}(\bar{\Omega},\mathcal{I};\tau)))c_1(\tau)
\end{equation*}
and further
\begin{align*} 
\bm{H_1}(\bar{\Omega},\mathcal{I};\tau)
= \bm{H}(\bar{\Omega} + \iota(\bm{q}(\bar{\Omega},\mathcal{I};\tau)),\bm{I}(\bar{\Omega},\mathcal{I};\tau);\tau)  
- \mathcal{I}^\frac{3-\alpha}{2(\alpha+1)}\bm{f}(\bar{\Omega} + \iota(\bm{q}(\bar{\Omega},\mathcal{I};\tau)))c(\tau).
\end{align*}
Then the relation $H_1(\varphi,\mathcal{I};\tau) = \bm{H_1}(\iota_{\bar{\Theta}}(\varphi),\mathcal{I};\tau)$ holds for all $(\varphi,\mathcal{I};\tau) \in \RN \times [\mathcal{I}_{**},\infty)\times \RN$.

\section{Proof of the main result} \label{section main result}

Let us recall equation (\ref{DGL 0}), that is
\begin{equation*} 
\ddot x + \lvert x \rvert^{\alpha-1}x = p(t), \;\; \alpha \geq 3,
\end{equation*}
where $p \in \mathcal{C}^4_b(\RN)$.
Denote by $x(t;v_0,t_0)$ the solution of this equation to the initial condition 
\begin{equation*}
x(t_0)=0, \;\; \dot{x}(t_0)=v_0.
\end{equation*}
Now set $\mathcal{I}^* = \max\{4 \mathcal{I}_{**}, (2\kappa_1)^{\frac{\alpha+3}{2}} \}$ and define $v_* = -2 \sqrt{\mathcal{I}^*}$. Thus $v_*$ is a constant depending only upon $\alpha, \omega$ and $\lVert \bm{p} \lVert_{\mathcal{C}^k(\TT^N)}$. Then, consider the function $\psi$ that maps the initial values $(v_0,t_0)\in (-\infty,v_*) \times \RN$ to $(v_1,t_1)$, where $v_1 = \dot{x}(t_1;v_0,t_0)$ and 
\begin{equation*}
t_1 = \inf \{s \in (t_0,\infty): x(s;v_0,t_0)=0, \; \dot{x}(s;v_0,t_0)<0 \}.
\end{equation*}
We want to show, that this map is well defined on $(-\infty,v_*) \times \RN$. To this end, let us recall the transformations of section \ref{section transformations}:
\begin{equation*}
(x,v;t) \overset{\mathcal{R}}{\rightarrow} (\bar{\vartheta},r;t) \hookrightarrow  (\vartheta,r;t) \overset{\mathcal{S}}{\rightarrow} (\phi,I;\tau) \overset{\mathcal{T}}{\rightarrow} (\varphi,\mathcal{I};\tau)
\end{equation*}
Since $x=0$ and $v<0$ corresponds to $\bar{\vartheta} = \pi/2$ and therefore $\vartheta = \tau \in \{\pi/2 + 2\pi \IN\}$, one variable becomes redundant if we stay on the lower $y$-axis. Therefore we consider restrictions of the transformation maps onto some $2$-dimensional subspaces, namely: 
\begin{align*}
&\mathcal{R}_0:(-\infty,0) \times \RN \to (0,\infty) \times \RN, \;\; \mathcal{R}_0(v,t) = (\pi_2(\mathcal{R}(0,v;t)),\pi_3(\mathcal{R}(0,v;t))), \\
&\mathcal{S}_0:[r_*,\infty) \times \RN \to \RN \times (0,\infty), \;\; \mathcal{S}_0(r,t) = (\pi_1(\mathcal{S}(\pi/2,r;t)),\pi_2(\mathcal{S}(\pi/2,r;t))), \;\; \text{and} \\
&\mathcal{T}_0:  \RN \times [I_{**},\infty) \to \RN \times (0,\infty), \;\; \mathcal{T}_0(\phi,I) = (\pi_1(\mathcal{T}(\phi,I;\pi/2)),\pi_2(\mathcal{T}(\phi,I;\pi/2))),
\end{align*}
where $\pi_j:\RN^3 \to \RN$ denotes the projection on to the $j$-th component.
For $\lvert v_0 \rvert$ sufficiently large, let $\mathcal{R}_0(v_0,t_0) = (r_0,t_0)$, $\mathcal{S}_0(r_0,t_0) = (\phi_0,I_0)$ and $\mathcal{T}_0(\phi_0,I_0)=(\varphi_0,\mathcal{I}_0)$. Plugging $x_0=0$ into (\ref{eq energy action}) and the definition of $I_0$ give us
\begin{equation} \label{eq v r I}
\frac{1}{2}v_0^2  = \kappa_1 r_0^{\frac{2(\alpha+1)}{\alpha+3}} = \mathcal{H}(\pi/2,r_0;t_0) = I_0,
\end{equation}
due to $c(\pi/2)=0$. Moreover, Theorem \ref{thm ham transformation} yields $2I_0 \geq \mathcal{I}_0 \geq \frac{I_0}{2}$. So in total we have
\begin{equation*}
\mathcal{I}_0 \geq \frac{I_0}{2} = \frac{v_0^2}{4}.
\end{equation*}
This implies $(\mathcal{T}_0 \circ \mathcal{S}_0 \circ \mathcal{R}_0) ((-\infty,v_*)\times \RN) \subset \RN \times (\mathcal{I}^*,\infty)$. For $\mathcal{I}_0>\mathcal{I}^*$ on the other hand, Lemma \ref{lem I abs} yields the existence of the corresponding solution $(\varphi,\mathcal{I})(\tau;\varphi_0,\mathcal{I}_0,\pi/2)$ to system (\ref{Ham 4 equations}) on $[\pi/2,5\pi/2]$ and guarantees $\mathcal{I}(\tau) >\mathcal{I}_{**}$. But since
\begin{equation} \label{inklusion transform back}
\mathcal{T}_0^{-1}(\RN\times[\mathcal{I_{**}},\infty))  \subset \RN\times[{I_*},\infty) \text{ and } \mathcal{S}_0^{-1}(\RN\times[{I_{*}},\infty))  \subset \RN\times[{r_*},\infty),
\end{equation}
$(\varphi,\mathcal{I})$ can be transformed back to a solution $x(t;v_0,t_0)$ of the original problem (\ref{DGL 0}) and
\begin{equation*}
(v_1,t_1) = (\mathcal{R}_0^{-1} \circ \mathcal{S}_0^{-1} \circ \mathcal{T}_0^{-1})(\varphi(5\pi/2),\mathcal{I}(5\pi/2))
\end{equation*}
has all the desired properties. Thus, $\psi$ can be equivalently defined in the following way:
\begin{equation*} \label{def psi}
\psi:(-\infty,v_*) \times \RN \to (-\infty,0)\times \RN, \;\;  \psi =   (\mathcal{R}_0^{-1} \circ \mathcal{S}_0^{-1} \circ \mathcal{T}_0^{-1})  \circ \Phi \circ (\mathcal{T}_0 \circ \mathcal{S}_0 \circ \mathcal{R}_0)
\end{equation*}
Finally we are ready to state and prove the main theorem.
\begin{thm} \label{main theorem}
	Let $\bm{p}\in \mathcal{C}^4(\TT^N)$ generate the family of forcing functions
	\begin{equation*} 
	p_{\bar{\Theta}}(t)=\bm{p}(\overline{\theta_1+t\omega_1}, \ldots, \overline{\theta_N+t\omega_N}), \;\; \bar{\Theta}=(\bar\theta_1,\ldots,\bar\theta_N)\in\TT^N,
	\end{equation*}
	for fixed rationally independent frequencies $\omega_1,\ldots,\omega_N>0$. Let $(v_n,t_n)_{n \in J} = (\psi^n(v_0,t_0))_{n \in J}$ denote a generic orbit associated to system (\ref{DGL 0}) with $p$ replaced by $p_{\bar{\Theta}}$. The escaping set $\mathcal{E}_{\bar{\Theta}}$ consists of those initial values $(v_0,t_0) \in \RN^2$ such that
	\begin{enumerate} [label=(\alph*)]
		\item $v_0<v_*$,
		\item the corresponding orbit satisfies $\NN \subset J$, and 
		\item $\lim_{n\to\infty} v_n = -\infty$.
	\end{enumerate}
	Then, for almost all $\bar{\Theta} \in \TT^N$, the set $\mathcal{E}_{\bar{\Theta}}$ has Lebesgue measure zero.
\end{thm}
The proof will be divided in two parts: First we are going to construct a measure-preserving embedding suitable for Theorem \ref{thm escaping set nullmenge}, which basically translates into the successor map $\Phi$ of system (\ref{Ham 4 equations}). Therefore it can be shown that for almost all $\bar{\Theta} \in \TT^N$ the corresponding escaping set $E_{\bar{\Theta}}$ has Lebesgue measure zero. In the second part we will prove that initial values in $\mathcal{E}_{\bar{\Theta}}$ correspond to points in $E_{\bar{\Theta}}$ and conclude $\lambda^2(\mathcal{E}_{\bar{\Theta}})=0$ for almost all $\bar{\Theta} \in \TT^N$.

\subsection{Escaping orbits of the transformed system}

For $\bar{\Theta}\in \TT^N$, denote by $(\varphi_{\bar{\Theta}}(\tau;\varphi_0,\mathcal{I}_0,\pi/2),\mathcal{I}_{\bar{\Theta}}(\tau;\varphi_0,\mathcal{I}_0,\pi/2))$ the solution to system (\ref{Ham 4 equations}) with initial data $\varphi(\pi/2)= \varphi_0, \; \mathcal{I}(\pi/2)= \mathcal{I}_0$ and forcing function $p = p_{\bar{\Theta}}$. 
Furthermore, we will write  $(\varphi(\tau;\bar{\Theta}_0,\mathcal{I}_0,\pi/2),\mathcal{I}(\tau;\bar{\Theta}_0,\mathcal{I}_0,\pi/2))$ for the solution of
\begin{align} \label{DGL on torus}
\varphi' = \p_{\mathcal{I}} \bm{H_1}(\bar{\Theta}_0 +\iota(\varphi),\mathcal{I};\tau), 
\;\; \mathcal{I}'= -\p_\omega \bm{H_1}(\bar{\Theta}_0 +\iota(\varphi),\mathcal{I};\tau), 
\end{align}
with $\bm{H_1}$ defined as in the last section and the initial values $\varphi(\pi/2)= 0, \; \mathcal{I}(\pi/2)= \mathcal{I}_0$. \\
If we have $\bar{\Theta}_0 = \iota_{\bar{\Theta}}(\varphi_0)$ these solutions meet the identity
\begin{equation} \label{Eq gleiche Losungen}
(\varphi_{\bar{\Theta}}(\tau;\varphi_0,\mathcal{I}_0,\pi/2),\mathcal{I}_{\bar{\Theta}}(\tau;\varphi_0,\mathcal{I}_0,\pi/2)) = (\varphi_0 +\varphi(\tau;\bar{\Theta}_0,\mathcal{I}_0,\pi/2),\mathcal{I}(\tau;\bar{\Theta}_0,\mathcal{I}_0,\pi/2)).
\end{equation}
Let $F,G:\mathcal{D} \to \RN$, where $\mathcal{D}=\TT^N \times (\mathcal{I}^*,\infty)$, be defined by
\begin{equation*}
F(\bar{\Theta}_0,\mathcal{I}_0) 
= \int_{\frac{\pi}{2}}^{\frac{5\pi}{2}} \varphi'(\tau;\bar{\Theta}_0,\mathcal{I}_0,\pi/2) \, d\tau
\end{equation*}
and
\begin{equation*}
G(\bar{\Theta}_0,\mathcal{I}_0) 
= \int_{\frac{\pi}{2}}^{\frac{5\pi}{2}} \mathcal{I}'(\tau;\bar{\Theta}_0,\mathcal{I}_0,\pi/2) \, d\tau,
\end{equation*}
respectively. And consider $g:\mathcal{D}\to \TT^N\times(0,\infty)$ given by
\begin{equation*}
g(\bar{\Theta}_0,\mathcal{I}_0) = (\bar{\Theta}_0 + \iota(F(\bar{\Theta}_0,\mathcal{I}_0) ), \mathcal{I}_0 + G(\bar{\Theta}_0,\mathcal{I}_0) ).
\end{equation*}
Then $F$ and $G$ are continuous, since the solution of (\ref{DGL on torus}) depends continuously upon the initial condition and the parameter $\bar{\Theta}_0$. Therefore $g$ has special form (\ref{embedding form}). The corresponding family of maps of the plane $\{g_{\bar{\Theta}}\}_{\bar{\Theta}\in \TT^N}$ as in (\ref{planar maps}) is
\begin{align*}
&g_{\bar{\Theta}}:D_{\bar{\Theta}}\subset \RN \times (0,\infty)\to \RN \times (0,\infty),\\
&g_{\bar{\Theta}}(\varphi_0,\mathcal{I}_0) = (\varphi_0 + F(\bar{\Theta}+\iota(\varphi_0), \mathcal{I}_0), \mathcal{I}_0 + G(\bar{\Theta}+\iota(\varphi_0), \mathcal{I}_0)),
\end{align*}
where $D_{\bar{\Theta}} = (\iota_{\bar{\Theta}}\times \id)^{-1}(\mathcal{D}) = \RN \times (\mathcal{I_{**}},\infty)$. Because of (\ref{Eq gleiche Losungen}) these maps coincide with the successor map $\Phi$ of (\ref{Def successor}) for the forcing function $p_{\bar{\Theta}}$. \\
The injectivity of $g$ is a consequence of the unique resolvability of the initial value problem
\begin{align*}
\varphi' = \p_{\mathcal{I}} \bm{H}_1(\bar{\Theta}_1 +\iota(\varphi),\mathcal{I};\tau), 
\;\; \mathcal{I}'= -\p_\omega \bm{H}_1(\bar{\Theta}_1 +\iota(\varphi),\mathcal{I};\tau), 
\end{align*}
with $\varphi(5\pi/2)= 0$ and $\mathcal{I}(5\pi/2)= \mathcal{I}_1$, where $(\bar{\Theta}_1,\mathcal{I}_1)=g(\bar{\Theta}_0,\mathcal{I}_0)$.\\
It remains to proof that $g$ is measure-preserving. Since the maps $g_{\bar{\Theta}}$ correspond to a Hamiltonian flow, Liouville's theorem yields $\det J_{g_{\bar{\Theta}}}(\varphi_0,\mathcal{I}_0) = 1$, i.e.
\begin{equation*}
1=  \left(1+\p_\omega F(\bar{\Theta}+\iota(\varphi_0), \mathcal{I}_0)\right)\left(1+\p_{\mathcal{I}_0}G(\bar{\Theta}+\iota(\varphi_0), \mathcal{I}_0)\right)
- \left(\p_{\mathcal{I}_0}F(\bar{\Theta}+\iota(\varphi_0), \mathcal{I}_0)\right)\left(\p_\omega G(\bar{\Theta}+\iota(\varphi_0), \mathcal{I}_0)\right).
\end{equation*}
Since this is valid for all  $\bar{\Theta} \in \TT^N$ and all of $D_{\bar{\Theta}}$, it follows 
\begin{equation*}
1=  \left(1+\p_\omega F \right) \left(1+\p_{\mathcal{I}_0}G\right)
- \left(\p_{\mathcal{I}_0}F\right) \left(\p_\omega G\right).
\end{equation*}
However, as shown in \cite[Lemma~3.3]{kunze_ortega_ping_pong} this implies that the map $g$ is orientation- and measure-preserving. \\
Hence we have shown that $g$ is a measure-preserving embedding. Now, we have to find functions $W,k$ as described in Theorem \ref{thm escaping set nullmenge}. Since $C$ from Lemma \ref{lem adiabatic invariant} depends only upon $\lVert f \rVert_{\mathcal{C}_b^4(\RN)}$, this constant is uniform in $\bar{\Theta} \in \TT^N$. Therefore, if we take $W(\bar{\Theta}_0,\mathcal{I}_0)=\mathcal{I}_0$, Lemma \ref{lem adiabatic invariant} implies
\begin{equation*}
W(g(\bar{\Theta}_0,\mathcal{I}_0)) - W(\bar{\Theta}_0,\mathcal{I}_0) = \mathcal{I}_1 - \mathcal{I}_0 \leq k(\mathcal{I}_0),
\end{equation*}
where $k(\mathcal{I}_0) = C \mathcal{I}_0^{b_\alpha}$ with $C$ as mentioned above and $b_\alpha<0$ from Theorem \ref{thm ham transformation}. That way $W$ and $k$ meet all demanded criteria. Thus the measure-preserving embedding $g$ satisfies all conditions of Theorem \ref{thm escaping set nullmenge} and we are finally ready to apply it. This gives us $\lambda^2(E_{\bar{\Theta}})=0$ for almost all $\bar{\Theta} \in \TT^N$ for the escaping set
\begin{equation*}
{E_{\bar{\Theta}}}= \{(\varphi_0,\mathcal{I}_0)\in{D}_{{\bar{\Theta}},\infty}: \lim_{n\to\infty}\mathcal{I}_n=\infty  \},
\end{equation*}
where ${D}_{{\bar{\Theta}},\infty}$ is the set of initial conditions leading to complete forward orbits of $g_{\bar{\Theta}}$ as described in subsection \ref{subsection quasi-periodic functions}. \\

\subsection{Undoing the transformations}

Now let $\bar{\Theta}\in \TT^N$ be fixed and consider the set
\begin{equation*}
\tilde{E}_{\bar{\Theta}} = (\mathcal{R}_0^{-1} \circ \mathcal{S}_0^{-1} \circ \mathcal{T}_0^{-1})( {E_{\bar{\Theta}}} ).
\end{equation*}
Our strategy will be to show $\lambda^2(\tilde{E}_{\bar{\Theta}})=0$ and $\mathcal{E}_{\bar{\Theta}} \subset \tilde{E}_{\bar{\Theta}}$. \\
Since $\mathcal{T}^{-1}(\cdot,\cdot;\tau)$ is symplectic for all $\tau \in \RN$, it follows $\lambda^2(\mathcal{T}_0^{-1}(E_{\bar{\Theta}})) = \lambda^2({E_{\bar{\Theta}}}) = 0$. Also, due to (\ref{inklusion transform back}) it is $\mathcal{S}_0^{-1}(\mathcal{T}_0^{-1}( {E_{\bar{\Theta}}} )) \subset [{r_*},\infty)\times \RN$. But then, since $\mathcal{S}_0(r_0,t_0) = (t_0,\mathcal{H}(0,r_0;t_0))$, the definition of $r_*$ (\ref{Mono abs}) yields $\left\lvert \det J_{\mathcal{S}_0^{-1}}(\phi,I) \right\rvert = \left\lvert \p_r \mathcal{H} (0,\mathcal{S}_0^{-1}(\phi,I)) \right\rvert^{-1} \leq 1$ on $\mathcal{T}_0^{-1}( {E_{\bar{\Theta}}} )$. Hence
\begin{equation*}
\lambda^2(\mathcal{S}_0^{-1}(\mathcal{T}_0^{-1}( {E_{\bar{\Theta}}} ))) \leq  \lambda^2(\mathcal{T}_0^{-1}({E_{\bar{\Theta}}})) = 0.
\end{equation*}
Finally, due to (\ref{eq v r I}),(\ref{abs I aus lemma existenzint}) and the definition of $\mathcal{I}^*$ we have
\begin{equation*}
\kappa_1 r_0^{\frac{2(\alpha+1)}{\alpha+3}} > \frac{\mathcal{I}^*}{2} \geq \frac{1}{2}(2\kappa_1)^{\frac{\alpha+3}{2}} \;\; \text{for} \;\; (r_0,t_0) \in \mathcal{S}_0^{-1}(\mathcal{T}_0^{-1}( {E_{\bar{\Theta}}} ))
\end{equation*}
and consequently $\left\lvert \det J_{\mathcal{R}_0^{-1}}(r_0,t_0) \right\rvert  =\left\lvert - \sqrt{2 \kappa_1} \frac{\alpha +1}{\alpha + 3} r_0^{-\frac{2}{\alpha+3}}\right\rvert <1$. Therefore we can conclude
\begin{equation*}
\lambda^2( \tilde{E}_{\bar{\Theta}} ) =  0.
\end{equation*}
It remains to show that $\mathcal{E}_{\bar{\Theta}} \subset \tilde{E}_{\bar{\Theta}}$. If $\mathcal{E}_{\bar{\Theta}}$ is empty there is nothing to prove. Otherwise let $(v_0,t_0) \in \mathcal{E}_{\bar{\Theta}}$ and consider the corresponding orbit
\begin{equation*}
(v_n,t_n) = \psi^n(v_0,t_0), \;\; n\in\NN.
\end{equation*}
Above we have demonstrated $\psi =   (\mathcal{R}_0^{-1} \circ \mathcal{S}_0^{-1} \circ \mathcal{T}_0^{-1})  \circ \Phi \circ (\mathcal{T}_0 \circ \mathcal{S}_0 \circ \mathcal{R}_0)$, which means
\begin{equation*}
(v_n,t_n) =  (\mathcal{R}_0^{-1} \circ \mathcal{S}_0^{-1} \circ \mathcal{T}_0^{-1}) (\varphi_n,\mathcal{I}_n)
\end{equation*}
with the notation as usual. Since also
\begin{equation*}
\mathcal{I}_n \geq \frac{v_n^2}{2}
\end{equation*}
holds for all $n \in \NN$, we can conclude $\mathcal{I}_n\to\infty$ as $n \to\infty$. But this implies $(\varphi_0,\mathcal{I}_0) \in E_{\bar{\Theta}}$ and therefore completes the proof.

\appendix

\bibliographystyle{alpha}
\bibliography{bibliography}

\end{document}